\documentclass[a4paper,11pt, oneside]{amsart}
\usepackage[english]{babel}
\usepackage{amsthm}
\usepackage{amssymb}
\usepackage{amsfonts}
\usepackage{amsmath}
\usepackage{amsthm}
\usepackage[all]{xy}
\usepackage{geometry}
\usepackage{ae,aecompl}
\usepackage{eurosym}
\usepackage{verbatim}
\usepackage{mathrsfs}
\usepackage{caption}
\usepackage{url}
\usepackage{tikz}
\usepackage{hyperref}

\newcounter{stepcounter}
\newtheoremstyle{smallcaps}
    {3pt}                    
    {3pt}                    
    {\itshape}                   
    {}                           
    {\sc}                   
    {.}                          
    {.5em}                       
    {}  
    
\newtheoremstyle{smallcapsdef}
    {3pt}                    
    {3pt}                    
    {}                   
    {}                           
    {\sc}                   
    {.}                          
    {.5em}                       
    {}  

\theoremstyle{plain}

\newtheorem{thm}{Theorem}[section]

\newtheorem{lem}[thm]{Lemma}
\newtheorem{prop}[thm]{Proposition}
\newtheorem{cor}[thm]{Corollary}

\theoremstyle{definition}

\newtheorem{eg}[thm]{Example}
\newtheorem{defn}[thm]{Definition}

\newtheorem{remark}[thm]{Remark}

\date{}

\newcommand\bit{\begin{itemize}}
\newcommand\eit{\end{itemize}}
\newcommand\bet{\begin{enumerate}}
\newcommand\eet{\end{enumerate}}
\newcommand\ed{\end{document}}

\DeclareFontFamily{U}{mathx}{\hyphenchar\font45}
\DeclareFontShape{U}{mathx}{m}{n}{
      <5> <6> <7> <8> <9> <10>
      <10.95> <12> <14.4> <17.28> <20.74> <24.88>
      mathx10
      }{}
\DeclareSymbolFont{mathx}{U}{mathx}{m}{n}
\DeclareFontSubstitution{U}{mathx}{m}{n}
\DeclareMathAccent{\widecheck}{0}{mathx}{"71}
\DeclareMathAccent{\wideparen}{0}{mathx}{"75}

\renewcommand{\a}{\alpha}

\newcommand{\e}{\varepsilon}
\newcommand{\f}{\varphi}
\renewcommand{\k}{\kappa}

\newcommand\s{\sigma}

\newcommand\w{\omega}

\newcommand\Om{\Omega}

\newcommand\del{\partial}
\newcommand\adel{\ol{\partial}}
\newcommand\DEL{\Delta}

\newcommand\bC{{\mathbb C}}

\newcommand\bR{{\mathbb R}}

\newcommand\F{{\mathcal F}}

\renewcommand{\O}{\mathcal{O}}

\newcommand\co{\mathrm{co}}

\newcommand\exd{\mathrm{d}}

\newcommand\unit{\mathrm{U}}
\newcommand\counit{\mathrm{C}}

\newcommand\id{\mathrm{id}}

\newcommand\hol{^{(1,0)}}
\newcommand\ahol{^{(0,1)}}

\newcommand\oby{\otimes}

\newcommand\wed{\wedge}
\newcommand\sseq{\subseteq}

\def\qbinom#1#2{\ensuremath{\left[\kern-.3em\left[\genfrac{}{}{0pt}{}{#1}{#2}\right]\kern-.3em\right]_q}}

\newcommand\ol{\overline}

\newcommand\bs{\backslash}
\newcommand\mto{\mapsto}

\newcommand\alg{algebra~}

\newcommand\wrt{with respect to~}

\newcommand{\EE}{\mathcal{E}}

\def \coloneqq{:=}

\usepackage{tikz}
\usetikzlibrary{decorations.pathreplacing}

\usepackage{ytableau}

\usepackage[english]{babel}

\author{Biswarup Das}
\address{Laboratory of Advanced Combinatorics and Network Applications,
Department of Applied Mathematics, Moscow Institute of Physics and Technology, Moscow, Russia}
\address{Instytut Matematyczny, Uniwersytet Wroc\l{}awski, pl.Grunwaldzki 2/4, 50-384 Wroc\l{}aw, Poland}
\email{biswarup.das@math.uni.wroc.pl}

\author[R. \'O Buachalla]{R\'eamonn \'O Buachalla}
\address{Mathematical Institute of Charles University, Sokolovsk\'a 83, Prague, Czech Republic} \email{obuachalla@karlin.mff.cuni.cz}

\author{Petr Somberg}
\address{Mathematical Institute of Charles University, Sokolovsk\'a 83, Prague, Czech Republic} \email{somberg@karlin.mff.cuni.cz}

\title[{\bf Spectral Gaps for Noncommutative  Dolbeault--Diracs}]{{\bf  Spectral Gaps for Twisted Dolbeault--Dirac Operators over the Irreducible Quantum Flag Manifolds
}}

\thanks{{\tiny  R\'OB acknowledges FNRS support through  a postdoctoral fellowship within the framework of the MIS Grant ``Antipode'' grant number F.4502.18. R\'OB is supported by the Charles University PRIMUS grant Spectral Noncommutative Geometry of Quantum Flag
Manifolds PRIMUS/21/SCI/026. P. Somberg acknowledges support of the grant GACR 19-28628X}
 }

\begin{document}

\maketitle

\begin{abstract}
We show that tensoring the Laplace and Dolbeault--Dirac operators of a K\"ahler structure (with closed integral) by a negative Hermitian holomorphic module, produces operators with spectral gaps around zero. The proof is based on the recently established  Akizuki--Nakano identity of a noncommutative K\"ahler structure. This general framework is then applied to the Heckenberger--Kolb calculus of the irreducible quantum flag manifolds $\O_q(G/L_S)$, and it is shown that twisting their  Dirac and Laplace operators by negative line bundles produces a spectral gap, for $q$ sufficiently close to $1$. The main technical challenge in applying the framework is to establish positivity of the quantum Fubini--Study metric of $\O_q(G/L_S)$. Importantly, combining positivity with the noncommutative hard Lefschetz theorem, it is additionally observed that the even degree de Rham cohomology groups of the Heckenberger--Kolb calculi do not vanish.
\end{abstract}


\section{Introduction}


The spectral theory of Dirac and Laplace operators plays a foundational role in noncommutative geometry, and in particular in Connes' spectral triple approach to the subject. In general, interest focuses on the spectral growth of noncommutative Dirac operators, with a view to showing that their bounded transform has compact resolvent and hence  gives a $K$-homology class. In this paper we focus instead on noncommutative generalisations of spectral gaps, which we take to mean an open interval around zero containing no non-zero eigenvalues of the operator. In the classical setting spectral gaps are closely  connected to the geometry of the manifold, as exemplified by Cheeger's inequality  \cite{CheegerIneq} and Buser's inequality \cite{BuserIneq}. Moreover, spectral gaps are intimately related to the rate of decay of the associated  heat kernels. A standard way of producing a spectral gap is through a Weitzenb\"ock identity, where two different Laplacians are shown to differ by a positive curvature operator. The lowest eigenvalue of the curvature operator provides the lower bound for the eigenvalues.
An important example of this procedure is the lower bound the Schr\"odinger--Lichnerowicz operator gives for the square of the eigenvalues of the Dirac operator of a compact Riemannian spin manifold. Explicitly, the lower bound is given by $R_{\mathrm{min}}/4$, where $R_{\mathrm{min}}$ denotes the minimum of the scalar curvature, a bound that would later be refined by Friedrich \cite[\textsection 5]{FriedrichDirac}.


The Weitzenb\"ock  identity of particular interest to us here is the Akizuki--Nakano identity of a  K\"ahler manifold, which presents the difference of the holomorphic  and anti-holomorphic Laplcians, which have been twisted by an Hermitian holomorphic vector bundle, in terms of the commutator of the Chern curvature and the twisted dual Lefschetz map. \emph{Noncommutative} K\"ahler structures were recently introduced in \cite{MMF3} as a framework for exploring  noncommutative K\"ahler geometry on differential calculi. This framework allows for a direct noncommutative generalisation of a surprising number of classical structures, such as Lefschetz decomposition, K\"ahler--Hodge theory, and the K\"ahler identities, among many other results. It was subsequently shown that these structures could be twisted by a noncommutative Hermitian holomorphic module (in the sense of Beggs and Majid) and that one could produce a noncommutative generalisation of the Kodaira vanishing theorem, and most importantly from our point of view, a noncommutative Akizuki--Nakano identity. In this article, we use this identity to produce a spectra gap for the Dirac and Laplace operators which have been twisted by a negative, or anti-ample, Hermitian holomorphic module.

%
%


The motivating family of examples underlying the theory of noncommutative K\"ahler structures is the irreducible quantum flag manifolds. This very important class of quantum homogeneous spaces is exceptional, among the Drinfeld--Jimbo quantum spaces, as having a reasonably well-understood noncommutative differential geometry. They have  an essentially unique covariant $q$-deformation of their classical de Rham complex, the celebrated  Heckenberger and Kolb calculi. These differential graded algebras were later shown to possess a covariant noncommutative  K\"ahler structure unique up to scalar multiple \cite{MMF3,MarcoConj}. 
The daunting task of explicitly calculating the point spectrum of the associated Dolbeault--Dirac and Laplace operators has been accomplished for the special cases of quantum projective space $\O_q(\mathbb{CP}^n)$ and the $B_2$ irreducible quantum flag manifold \cite{DGOBW}, which is to say, the quantum quadric $\O_q(\mathbf{Q}_{5})$. Moreover, from this explicit spectral description one can conclude that the Dolbeault--Dirac operators give spectral triples, which is to say, Connes and Moscovici's notion of a noncommutative  Riemannian manifold. Similar constructions of spectral triples were undertaken in \cite{DSPodles, DDLCP2,DDCPN,MatassaCPn} for the case of $\O_q(\mathbb{CP}^n)$, and in \cite{MatassaParth} for the quantum Lagrangian Grassmannian $\O_q(\mathbf{L}_2)$ \cite{MatassaParth}.


Here we take an different approach and show that twisting the Dirac and Laplace operators by negative relative line modules gives operators with a spectral gap. In \cite{DOKSS} it was shown that the line bundles $\EE_k$ over $\O_q(G/L_S)$, which are indexed by the integers  $k \in \mathbb{Z}$, are positive if $k >0$ and negative if $k < 0$. This allows us to twist the Dolbeault--Dirac and Laplace operators to produce  spectral gaps. We highlight that  this is done using a conceptual argument, without any recourse to explicit calculation, and shows how the $q$-deformed spectral behaviour of these operators is moulded by the underlying complex geometry of the differential calculi. The main technical obstacle to the application of our general framework, is to show that the quantum Fubini--Study metrics of the quantum flag manifolds are positive definite. The question of positivity of the associated metrics has thus far only been addressed for the special case of quantum projective space. Here we show that if the deformation parameter lies in a sufficiently small open interval around $1$, then positivity is indeed satisfied. 

Positivity has important cohomological consequence. It allows us to apply the noncommutative hard Lefschetz theorem established in \cite{MMF3} and to show that the even de Rham cohomology groups $H^{2k}$ do not vanish, just as in the classical case.  This demonstrates that the Heckenberger--Kolb cohomologies do not suffer from the unpleasant dimension drop phenomenon occurring in cyclic cohomology,  the standard noncommutative generalisation of de Rham cohomology.

In the accompanying work of the authors \cite{DDFO}, we use the spectral gaps produced this paper to construct Fredholm operator completions of our twisted Dolbeault--Dirac operators. Moreover,  we calculate the index of the operator explicitly using the recently established  Borel--Weil theorem \cite{KMOS}, and Kodaira vanishing theorem \cite{OSV}. 

\subsection{Summary of the Paper}

The paper is organised as follows: In \textsection 2 we recall from \cite{MMF3} the necessary basics of Hermitian and K\"ahler structures, Hermitian holomorphic vector bundles, and positivity.

In \textsection 3  we establish the analogue of the  Akizuki--Nakano identitiy for Laplacians. We then show that these identities imply the existence of spectral gaps for the twisted Dolbeault--Dirac and Laplace operators.  

In \textsection 4 we recall the basic definitions and results of covariant differential calculi and K\"ahler structures over quantum homogeneous spaces. We prove  an alternative characterisation of covariant Hermitian metrics using Takeuchi's equivalence, and show that a covariant Hermitian structure gives an Hermitian module.

In \textsection 5  we present our  motivating family of examples, the irreducible quantum flag manifolds $\O_q(G/L_S)$. We recall the covariant noncommutative K\"ahler structure of each $\O_q(G/L_S)$, and show that its associated metric is positive definite, for $q$ sufficiently close to $1$. As an application, we observe non-vanishing of the central Dolbeault cohomology groups. We observe that, by the  general results of \textsection 3 and \textsection 4, twisting by a  relative line module $\EE_{-k}$ produces a spectral gap for the Dolbeault--Dirac  and Laplace operators. 

We finish with two appendices: In the first we present the monoidal version of Takeuchi's equivalence in the form best suited to the paper. In the second we discuss $q$-deformed spin-Dirac operators, observing that they admit spectral gaps.

\subsubsection*{Acknowledgements:} The authors would like to thank Karen Strung, Branimir \'{C}a\'{c}i\'c, Elmar Wagner,  Fredy D\'iaz Garc\'ia, Andrey Krutov,  Matthias Fischmann, Adam--Christiaan van Roosmalen, and Jan \v{S}t\!'ov\'i\v{c}ek, for many useful discussions during the preparation of this paper.  The second author would like to thank IMPAN Wroc\l{}aw for hosting him in November 2017, and would also  like to thank Klaas Landsman and the Institute for Mathematics, Astrophysics and Particle Physics, Radboud University, Nijmegen for hosting him in the winter of 2017 and 2018.

\section{Preliminaries on Hermitian Structures}

We recall the basic definitions and results for differential calculi, as well as complex, Hermitian, and K\"ahler structures. For a more detailed introduction see \cite{MMF2}, \cite{MMF3}, and references therein. For an excellent presentation of classical complex and K\"ahler geometry see \cite{HUY}.

\subsection{Differential Calculi}

A {\em differential calculus} $\big(\Om^\bullet \simeq \bigoplus_{k \in \mathbb{Z}_{\geq 0}} \Om^k, \exd\big)$ is a differential graded algebra (dg-algebra)  which is generated in degree $0$ as a dg-algebra, that is to say, it is generated as an algebra by the elements $a, \exd b$, for $a,b \in \Om^0$. 
For a given algebra $B$, a differential calculus {\em over} $B$ is a differential calculus such that $B = \Om^0$.  A differential calculus is said to be of {\em total degree} $m \in \mathbb{Z}_{\geq 0}$ if $\Om^m \neq 0$, and $\Om^{k} = 0$, for all $k > m$. A {\em differential $*$-calculus} over a $*$-algebra $B$ is a differential calculus over $B$ such that the \mbox{$*$-map} of $B$ extends to a (necessarily unique) conjugate linear involutive map $*:\Om^\bullet \to \Om^\bullet$ satisfying $\exd(\w^*) = (\exd \w)^*$, and 
\begin{align*}
\big(\w \wed \nu\big)^*  =  (-1)^{kl} \nu^* \wed \w^*, &  & \text{ for all } \w \in \Om^k, \, \nu \in \Om^l. 
\end{align*}We say that $\omega \in \Omega^{\bullet}$ is \emph{closed} if $\exd \omega = 0$, and \emph{real} if $\omega^* = \omega$. See \cite[\textsection 1]{BeggsMajid:Leabh} for a more detailed discussion of differential calculi.

\subsection{Complex Structures}

We now recall the definition of a complex structure \cite{KLvSPodles, BS}, which generalises the properties of the de Rham complex of a classical complex manifold \cite{HUY}.  An {\em almost complex structure} $\Om^{(\bullet,\bullet)}$, for a  differential $*$-calculus  $(\Om^{\bullet},\exd)$, is an $\mathbb{Z}^2_{\geq 0}$-\alg grading $\bigoplus_{(a,b)\in \mathbb{Z}^2_{\geq 0}} \Om^{(a,b)}$ for $\Om^{\bullet}$ such that
\begin{align*} 
\Om^k = \bigoplus_{a+b = k} \Om^{(a,b)}, & & \big(\Om^{(a,b)}\big)^* = \Om^{(b,a)}, & & \textrm{ for all } (a,b) \in \mathbb{Z}^2_{\geq 0}. 
\end{align*}
If the exterior derivative decomposes into a sum $\exd = \del + \adel$, for $\del$ a (necessarily unique) degree $(1,0)$-map, and $\adel$ a (necessarily unique) degree $(0,1)$-map, then we say that $\Om^{(\bullet,\bullet)}$ is a \emph{complex structure}. It follows that $\big(\Om^{(\bullet,\bullet)}, \del,\ol{\del}\big)$ is a double complex, which we call  the {\em Dolbeault double complex}, and that $\del$ and $\adel$ are $*$-maps. The {\em opposite} complex structure of an complex structure $\Om^{(\bullet,\bullet)}$ is the \mbox{$\mathbb{Z}^2_{\geq 0}$}-\alg grading  $\overline{\Om}^{(\bullet,\bullet)}$, defined by $\ol{\Om}^{(a,b)} := \Om^{(b,a)}$, for $(a,b) \in \mathbb{Z}_{\geq 0}^2$. See \cite[\textsection 1]{BeggsMajid:Leabh} or \cite{MMF2} for a more detailed discussion of complex structures.

\subsection{Hermitian and K\"ahler Structures} \label{subsection:HandKS}

We now present the definition of an Hermitian structure, as introduced in \cite[\textsection 4]{MMF3}, which generalises the properties of the fundamental form of an Hermitian metric.  

\begin{defn} An {\em Hermitian structure} $(\Om^{(\bullet,\bullet)}, \s)$ for a differential $*$-calculus $\Om^{\bullet}$, of even total degree $2n$,  is a pair  consisting of  a complex structure  $\Om^{(\bullet,\bullet)}$, and  a closed central real $(1,1)$-form $\kappa$, called the {\em Hermitian form}, such that, \wrt the {\em Lefschetz operator}
\begin{align*}
L_{\sigma}:\Om^\bullet \to \Om^\bullet,  & &   \w \mto \s \wed \w,
\end{align*}
isomorphisms are given by
\begin{align} \label{eqn:Liso}
L_{\sigma}^{n-k}: \Om^{k} \to  \Om^{2n-k}, & & \text{ for all } k = 0, \dots, n-1.
\end{align}
A \emph{K\"ahler structure} is an Hermitian structure $(\Om^{(\bullet,\bullet)}, \kappa)$ satisfying $\exd \kappa = 0$, and in this case we refer to $\kappa$ as the \emph{K\"ahler form}.
\end{defn}

In classical Hermitian geometry, the Hodge map of an Hermitian metric is related to the associated Lefschetz decomposition through the Weil formula (see \cite[Th\'eor\`eme 1.2]{Weil} or \cite[Proposition 1.2.31]{HUY}). In \cite[Definition 4.11]{MMF3} a noncommutative  generalisation of the Weil formula, based on a noncommutative generalisation of Lefschetz decomposition, was used to define a noncommutative Hodge map $\ast_{\sigma}$ for any noncommutative Hermitian structure.  We will not need the explicit definition here, and simply recall that $\ast_\s$ is a $*$-map, and that 
\begin{align*}
 \ast_{\s}(\Om^{(a,b)}) = \Om^{(n-b,n-a)}, & &  \ast_{\s}^2(\w) = (-1)^k \w,  & & \textrm{ for all } \w \in \Om^k.
\end{align*}

Reversing the classical order of construction, we now define a metric in terms of the Hodge map: The {\em metric} associated to the Hermitian structure $\big(\Om^{(\bullet,\bullet)},\s\big)$ is  the unique map $g_\s:\Om^\bullet \times \Om^\bullet \to B$  for which $g_{\sigma}\big(\Om^k, \Om^l\big) = 0$, for all $k \neq l$, and 
\begin{align*}
g_\s(\w, \nu) =  \ast_\s\big(\w \wed \ast_\s(\nu^*) \big), & &  \textrm{ for all } \w,\nu \in \Om^k.
\end{align*}
An important fact is that $g_{\sigma}$ is \emph{conjugate symmetric}, that is,
\begin{align*}
g_{\sigma}(\w, \nu) = g_{\sigma}(\nu, \w)^*, & & \text{ for all } \w, \nu \in \Om^\bullet.
\end{align*}
For a $*$-algebra $B$, we consider the \emph{cone of positive elements} 
\[ 
B_{\geq  0} := \mathrm{span}_{\mathbb{R}_{>0}}\left\{b_i^* b_i  \,|\, b_i \in B, \,  \in \mathbb{Z}_{\geq 0} \right\}\!.
\]
We denote the non-zero positive elements of $B$ by $B_{>0} := B_{\geq 0} \setminus \{0\}$. We say that an Hermitian structure $(\Omega^{(\bullet,\bullet)}, \sigma)$ is \emph{positive definite} if the associated metric $g_{\sigma}$ is \emph{positive definite}, which is to say, if $g_{\sigma}$ satisfies 
\begin{align*}
g_{\sigma}(\omega,\omega) \in B_{>  0}, & & \textrm{ for all } \omega \in \Omega^\bullet.
\end{align*}

\subsection{A Representation of $\frak{sl}_2$} \label{subsection:sl2}

As is readily verified \cite[Lemma 5.11]{MMF3}, the Lefschetz map $L_{\sigma}$ is adjointable on $\Omega^\bullet$ \wrt $g_{\sigma}$. 
Taking $L_{\sigma}$ and $\Lambda_{\sigma}$ together with the \emph{form degree counting operator}
\begin{align*}
H: \Omega^\bullet \to \Omega^\bullet, & & H(\omega) := (k-n) \omega,  ~~ \text{ for } \omega \in \Omega^k,
\end{align*}
we get the following commutator relations:
\begin{align*}
[H,L_{\sigma}] = 2 H, & & [L_{\sigma},\Lambda_{\sigma}] = H, & & [H,\Lambda_{\sigma}] = - 2  \Lambda_{\sigma}.
\end{align*}
Thus any K\"ahler structure gives a Lie algebra representation $T: \frak{sl}_2 \to  \frak{gl}\left(\Om^\bullet\right)$.

\subsection{Holomorphic Modules}

In this subsection we present the notion of a noncommutative holomorphic module, as  has been considered in various places, for example \cite{BS}, \cite{PolishSch}, and \cite{KLvSPodles}.  Motivated by the Serre--Swan theorem, we think of  finitely generated projective left $B$-modules as noncommutative generalisations of vector bundles. 
As we now recall, one can build on this idea to define noncommutative holomorphic modules via the classical Koszul--Malgrange characterisation of holomorphic bundles \cite{KoszulMalgrange}.  See \cite{OSV} for a more detailed discussion. 

For $\Omega^\bullet$ a differential calculus over an algebra $B$, and $\mathcal{F}$ a left $B$-module, a \emph{left connection} on $\F$ is a $\mathbb{C}$-linear map $\nabla:\mathcal{F} \to \Omega^1 \otimes_B \F$ satisfying 
\begin{align*}
\nabla(bf) = \exd b \otimes f + b \nabla f, & & \textrm{ for all } b \in B, f \in \F.
\end{align*}
With respect to a choice $\Omega^{(\bullet,\bullet)}$ of complex structure on $\Omega^{\bullet}$, a \emph{$(0,1)$-connection on $\mathcal{F}$} is a connection with respect to the differential calculus $(\Omega^{(0,\bullet)},\adel)$. Any connection can be extended to a map $\nabla: \Omega^\bullet \otimes_B \mathcal{F} \to   \Omega^\bullet \otimes_B \mathcal{F}$  by  defining 
\begin{align*}
\nabla(\omega \otimes f) =   \exd \omega \otimes f + (-1)^{|\omega|} \, \omega \wedge \nabla(f), & & \textrm{for } f \in \F, \, \omega \in \Omega^{\bullet},
\end{align*}
for a homogeneous form $\omega$ with degree  $|\omega|$.

The \emph{curvature} of a connection is the left $B$-module map $\nabla^2: \mathcal{F} \to \Omega^2 \otimes_B \mathcal{F}$. A connection is said to be {\em flat} if $\nabla^2 = 0$. Since $\nabla^2(\omega \otimes f) = \omega \wedge \nabla^2(f)$, a connection is flat if and only if  the pair $(\Omega^\bullet \otimes_B \F, \nabla)$ is a complex. 

\begin{defn}
For an algebra $B$, a \emph{holomorphic module over $B$} is a pair $(\mathcal{F},\adel_{\mathcal{F}})$, where  $\mathcal{F}$ is a finitely generated projective left $B$-module, and  $\adel_{\mathcal{F}}: \mathcal{F} \to \Omega^{(0,1)} \otimes_B \mathcal{F}$ is a flat $(0,1)$-connection, which we call the \emph{holomorphic structure} of $(\F, \adel_{\F})$. 
\end{defn}

\subsection{Hermitian Modules} \label{subsection:HVBS}

When $B$ is a $*$-algebra, we can also generalise the classical notion of an Hermitian metric for a vector bundle. For a left $B$-module $\F$, denote by $^\vee\!\F$ the dual module $^\vee\!\F := {}_B\mathrm{Hom}(\F, B)$, of left $B$-module maps, which is a right $B$-module with respect to the right  multiplication 
$
\phi b (f) := \phi(f) b, \text{for }  \phi \in \! \, ^\vee\!\F,  \text{ and } b \in B,  f \in \F. 
$
Moreover, we denote by $\overline{\F}$ the \emph{conjugate right $B$-module} of $\F$, as defined by the action $\overline{f}b = \overline{b^*f}$, for $f \in \F, b \in B$. If $\overline{\F}$ is a $B$-bimodule, then we have an analogously defined \emph{conjugate bimodule}.

\begin{defn} \label{defn:Hermitianmetric} 
An \emph{Hermitian module} over a $*$-algebra $B$ is a pair $(\F,h_{\F})$, where $\F$ is a finitely generated projective left $B$-module and  a right $B$-isomorphism $h_{\F}:\overline{\F} \to {}^{\vee}\!\F$, such that, for  the associated sesquilinear pairing, 
\begin{align*}
h_{\F}(-,-): \F \times \F \to B, & & (f,k) \mapsto h_{\F}(\overline{k})(f)
\end{align*}
it holds  that, for all $f,k \in \F$,
\begin{align*}
1. ~ h_{\F}(f,k) = h_{\F}(k,f)^*, & &  2. ~ h_{\F}(f,f) \in B_{>0}.
\end{align*}
\end{defn}
Consider next the map 
\begin{align*}
C_h: \Omega^1 \otimes_B \F \to {}^{\vee}\F \otimes_B \Omega^1, & & \omega \otimes f \mapsto h(\overline{f}) \otimes \omega^*,
\end{align*}
and the $B$-bimodule evaluation map 
$
\mathrm{ev}:\F \times \!{}^{\vee}\F \to B. 
$
Using these two maps, we can associate to any Hermitian module the \emph{Hermitian metric}
\begin{align*}
g_{\F}:= g_{\sigma} \circ (\id \otimes \mathrm{ev} \otimes \id) \circ (\id \times C_h): \Omega^{\bullet} \otimes_B \F \times \Omega^{\bullet} \otimes_B \F \to B.
\end{align*}
As shown in  \cite[\textsection 5]{OSV}, the metric is conjugate symmetric, which is to say
\begin{align} \label{eqn:conjsymm}
g_{\F}(\alpha,\beta) = g_{\F}(\beta,\alpha)^*, & & \textrm{ for all } \alpha,\beta \in \Omega^{\bullet} \otimes \F.
\end{align}
If the Hermitian structure is positive definite, then the metric $g_{\F}$ satisfies 
$
g_{\F}(\alpha,\alpha) \in B_{>0}$,  \textrm{ for every }  $\alpha \in \Omega^{\bullet} \otimes_B \F.
$
Finally, we note that the $\mathbb{Z}^2_{\geq 0}$ decomposition of $\Omega^{\bullet} \otimes_B \F$ is orthogonal with respect to $g_{\F}$.

\subsection{Inner Products, and Twisted Dolbeault--Dirac and Laplace Operators}

Let $\mathbf{f}:B \to \mathbb{C}$ be a \emph{state}, that is a linear map satisfying $\mathbf{f}(b^*b) > 0$, for all non-zero $b \in B$. As observed in \cite[\textsection 5.2]{OSV}, an inner product is then given by 
\begin{align*}
\langle \cdot,\cdot \rangle_{\F} := \mathbf{f} \circ g_{\F}: \Omega^{\bullet} \otimes_B \F \times \Omega^{\bullet} \otimes_B \F \to \mathbb{C}.
\end{align*}
Moreover, we can associate to $\mathbf{f}$ an \emph{integral}
\begin{align*}
\int := \mathbf{f} \circ \ast_{\sigma}: \Omega^{2n} \to \mathbb{C}.
\end{align*}
We say that $\int$ is \emph{closed} if the map
$
\int \circ \, \exd:\Omega^{n-1} \to \mathbb{C}
$
is equal to the zero map.

As shown in \cite[Proposition 5.15]{OSV}, for any Hermitian holomorphic vector bundle $(\F,\adel_{\F})$ over a K\"ahler structure $(B,\Omega^{\bullet},\Omega^{(\bullet,\bullet)},\sigma)$, with closed integral $\int$ with respect to a choice of state, the twisted differentials $\del_{\F}$ and $\adel_{\F}$ are  adjointable with respect to $\langle \cdot,\cdot\rangle_{\F}$.  The {\em $\F$-twisted holomorphic} and \emph{anti-holomorphic de Rham--Dirac operators} are respectively defined to be 
\begin{align*}
D_{\del_\F} := \del_\F + \del_\F^\dagger, & & D_{\adel_\F} :=  \adel_\F + \adel_\F^\dagger.
\end{align*}
Moreover, its  {\em $\F$-twisted  holomorphic} and \emph{ anti-holomorphic Laplace operators} are respectively defined to be 
\begin{align*}
\DEL_{\del_\F} \coloneqq D_{\del_{\F}}^2, & & \DEL_{\adel_\F} \coloneqq D_{\adel_{\F}}^2.
\end{align*}

For the twisted Dolbeault complex of a K\"ahler space, with closed integral, the following direct noncommutative generalisation of the Nakano identities was established in \cite[Theorem 7.6]{OSV}:
\begin{align} \label{eqn:NakanoIds}
 [L_\F,\del_\F] = 0, && [L_\F,\adel_\F] = 0,  & &  [\Lambda_\F, \del^\dagger_\F] = 0, & & [\Lambda_\F, \adel_\F^\dagger] = 0, \\
  [L_{\mathcal{F}},\del_{\F}^{\dagger}] = \mathbf{i}\adel_{\F},  &&   [L_\F,\adel_{\F}^{\dagger}] = - \mathbf{i} \del_\F,   &&  [\Lambda_{\F},  \del_{\F}] = \mathbf{i} \adel_{\F}^\dagger,  &&   [\Lambda_{\F},\adel_\F] = - \mathbf{i}\del^\dagger_{\F}.
\end{align}
We note that for the untwisted case these identities reduce to the noncommutative K\"ahler identities \cite[\textsection 7]{MMF3}. For a discussion of the classical situation see \cite[\textsection 5.3]{HUY} or \cite[\textsection VII.1]{JPDemaillyLeabh}.

As observed in \cite[Corollary 7.8]{OSV}, these identities imply that the classical relationship between the Laplacians $\DEL_{\del_{\F}}$ and $\DEL_{\adel_{\F}}$ carries over to the noncommutative setting. Note that, unlike the untwisted case \cite[Corollary 7.6]{MMF3}, the operators differ by a not necessarily trivial curvature operator:
\begin{align} \label{eqn:AkizNak}
\DEL_{\adel_{\F}} = \DEL_{\del_{\F}} + [\mathbf{i}\nabla^2, \Lambda_\F].
\end{align}
We call this identity the \emph{Akizuki--Nakano identity}.

\subsection{Nonvanishing of Even Degree Cohomology} \label{subsection:hardLef}

If the integral of a K\"ahler structure is closed with respect to a choice of state $\mathbf{f}:B \to \mathbb{C}$, and the untwisted Dolbeault--Dirac operator $D_{\adel}$ is diagonalisable, then the K\"ahler structure admits a direct noncommutative generalisation of Hodge decomposition \cite[\textsection 6]{MMF3}. Moreover, the classical definition of primitive cohomology and the hard Lefschetz theorem carry over directly to the noncommutative setting. Here we simply recall the following important consequence of the hard Lefschetz theorem: For a K\"ahler structure with closed integral and diagonalisable Dolbeault--Dirac operator,  
\begin{align*}
H^{2k} \neq 0, & & \textrm{ for all } k = 0, \dots, n, 
\end{align*}
where $H^{\bullet}$ is the de Rham cohomology of the differential calculus, and  $2n$ is its dimension.

\subsection{Chern Connections}

The following definition details a compatibility condition between connections and Hermitian structures. It is a direct generalisation of the classical condition \cite{HUY}. Let us first consider the extension of $h_{\F}(-,-)$ to the map
\begin{align*}
h_{\F}(-,-) = \wedge \circ (\id \otimes \mathrm{ev} \otimes \id) \circ (\id \otimes C_h): \Omega^{\bullet} \otimes_B \F  \times \Omega^{\bullet} \otimes_B \F \to \Omega^{\bullet}.
\end{align*}
Let $(\F,h_\F)$ be an Hermitian module, a connection $\nabla\colon \F \to \Om^1 \otimes_B\F$ is  \emph{Hermitian} if
\begin{align*}
\exd(h_\mathcal{F}(f,k)) = h_\F(\nabla(f), 1 \otimes {k}) + h_\F(1 \otimes  f,  {\nabla}({k})) && \textrm{ for all } f,k \in \F.
\end{align*} 
A \emph{holomorphic Hermitian module} is a triple $(\mathcal{F},h_\F,\adel_{\mathcal{F}})$ such that $(\F, h_\F)$ is an Hermitian module and $(\F, \adel_{\mathcal{F}})$ is a holomorphic module.  The following is shown in \cite{BeggsMajidChern}, see also \cite{OSV}.

For any  Hermitian holomorphic module $(\mathcal{F},h_\F,\adel_{\mathcal{F}})$, there exists a unique Hermitian connection $\nabla:\mathcal{F} \to  \Omega^1 \otimes_B \mathcal{F}$ satisfying
\begin{align*}
 \adel_{\F} = \big(\mathrm{proj}_{\Omega^{(0,1)}} \otimes \id\big) \circ  \nabla.
\end{align*}
Moreover we denote
$
\del_{\F} := \big(\mathrm{proj}_{\Omega^{(1,0)}} \otimes \id\big) \circ  \nabla,
$
and call $\nabla$ the \emph{Chern connection} of Hermitian holomorphic module $(\mathcal{F},h_\F, \adel_{\mathcal{F}})$.

\subsection{Positive Hermitian Holomorphic Modules}

We finish this subsection with the notion of positivity for a holomorphic Hermitian module. This directly generalises the classical notion of positivity, a property which is equivalent to ampleness \cite[Proposition 5.3.1]{HUY}. It was first introduced in \cite[Definition 8.2]{OSV} and requires a compatibility between Hermitian holomorphic modules and K\"ahler structures.
In the definition we use the following convenient notation 
\begin{align*}
L_{\F} := L_{\sigma} \otimes \id_{\F}, & & H_{\F} = H \otimes \id_{\F}, & & \Lambda_{\F} := \Lambda_{\sigma} \otimes \id_{\F}.
\end{align*}

\begin{defn}\label{defn:positiveNegativeVB}
Let $\Omega^{\bullet}$ be a differential calculus over a $*$-algebra $B$, and let $(\Omega^{(\bullet,\bullet)},\kappa)$ be a K\"ahler structure for $\Omega^{\bullet}$. 
An Hermitian  holomorphic module $(\mathcal{F},h_{\F}, \adel_{\mathcal{F}})$  is said to be \emph{positive}, written $\F > 0$, if there exists $\theta \in \mathbb{R}_{>0}$ such that the Chern connection $\nabla$ of ${\mathcal{F}}$ satisfies
\begin{align*}
\nabla^2(f) = -\theta \mathbf{i} L_{\F}(f) = -\theta \mathbf{i} \kappa \otimes f, & & \textrm{ for all } f \in \mathcal{F}.
\end{align*} 
Analogously, $(\mathcal{F}, h_{\F}, \adel_{\mathcal{F}})$ is said to be {\em negative}, written $\F <0$, if there exists $\theta \in \mathbb{R}_{>0}$ such that the Chern connection $\nabla$ of ${\mathcal{F}}$ satisfies
\begin{align*}
\nabla^2(f) = \theta \mathbf{i} L_{\F}(f) =  \theta \mathbf{i} \kappa \otimes f, & & \textrm{ for all } f \in \mathcal{F}.
\end{align*} 
\end{defn}

\section{A Spectral Gaps for Laplace and  Dolbeault--Dirac Operators}

In this section we prove the main general results of the paper, proving spectral gaps for twisted Laplace and Dolbeault--Dirac operators. We begin by proving that, just as in the classical case, the Chern--Laplacian differs  from the twisted holomorphic, and anti-holomorphic, Dolbeault--Dirac operators by a curvature operator. We then use this identity, along with the Akizuki--Nakano identity, to establish spectral gaps on the lower half of the Hodge diamond of a twisted K\"ahler structure, and on its anti-holomorphic subcomplex.  Thus we are able to conclude spectral behaviour for our Dirac operator from the underlying noncommutative complex geometry of the differential calculus.

\subsection{The Chern--Dirac and Chern--Laplace Operators}

Let $(\F,h,\adel_{\F})$ be an Hermitian holomorphic vector bundle over a noncommutative K\"ahler structure with closed integral. We observe that $\nabla$, the associated Chern connection of $\F$, is an adjointable operator, with adjoint given explicitly by $\nabla^{\dagger} := (\del_{\F} + \adel_{\F})^{\dagger} = \del_{\F}^{\dagger} + \adel_{\F}^{\dagger}$. In direct analogy with the untwisted case, we introduce the twisted de Rham--Dirac and twisted Laplace operators
\begin{align*}
D_{\nabla} := \nabla + \nabla^{\dagger}, & & \Delta_{\nabla} := \nabla \circ \nabla^{\dagger} + \nabla^{\dagger} \circ \nabla,
\end{align*}
In the untwisted case, where we denote our three Laplacians by $\Delta, \Delta_{\del}$, and $\Delta_{\del}$, we have  
\begin{align*}
\Delta = \Delta_{\del} + \Delta_{\adel} = 2\Delta_{\del} = 2\Delta_{\adel}.
\end{align*}
In the untwisted case, the generalisation of these equalities involves an additional curvature term, which is to say, we have a noncommutative Weitzenb\"ock identity.  We will prove this using the identities 
\begin{align} \label{eqn:anticommrels}
\del_{\F}\ol{\del}^\dagger_{\F} + \ol{\del}^\dagger_{\F} \del_{\F} = 0, & & \del^\dagger_{\F} \ol{\del}_{\F} + \ol{\del}_{\F}\del^\dagger_{\F} = 0,
\end{align}
which were established in  \cite[Corollary 7.7]{OSV} as an implication of the noncommutative Nakano identities.

\begin{prop}
For any K\"ahler  structure $(\Omega^{(\bullet,\bullet)},\kappa)$, with closed integral, and any Hermitian holomorphic vector bundle $(\F,h,\adel_{\F})$, the following identities hold on $\Omega^{\bullet}$:
\begin{align} \label{eqn:ChernAkizukiNak}
\Delta_{\nabla} = \Delta_{\del_{\F}} + \Delta_{\adel_{\F} }= 2\Delta_{\del_{\F}} - [\Lambda_{\F},\mathbf{i}\nabla^2] = 2\Delta_{\adel_{\F}} + [\Lambda_{\F},\mathbf{i}\nabla^2].
\end{align}
\end{prop}
\begin{proof}
We begin by expanding the expression for $\nabla \circ \nabla^{\dagger}$ as follows
\begin{align*} 
\nabla\circ \nabla^{\dagger} = & \, (\del_{\F} + \adel_{\F}) \circ (\del_{\F}^{\dagger} + \adel_{\F}^{\dagger})  =  \,   \del_{\F} \circ \del_{\F}^{\dagger} + \del_{\F} \circ \adel_{\F}^{\dagger} + \adel_{\F} \circ \del_{\F}^{\dagger} + \adel_{\F} \circ \adel_{\F}^{\dagger}.
\end{align*}
Recalling now the Nakano identities \eqref{eqn:NakanoIds}, we see that this expression is equal to 
\begin{align*}
 \del_{\F} \circ  \mathbf{i}[\Lambda_{\F},\adel_{\F}] - \del_{\F} \circ \mathbf{i}[\Lambda_{\F},\del_{\F}]  + \adel_{\F} \circ \mathbf{i}[\Lambda_{\F},\adel_{\F}] - \adel_{\F} \circ \mathbf{i}[\Lambda_{F},\del_{\F}].
 \end{align*}
Expanding the commutator brackets and regrouping gives us the expression
\begin{align*}
 \left(- \del_{\F} \circ \adel_{\F} + \adel_{\F} \circ \del_{\F} \right) \circ \mathbf{i} \Lambda_{\F} + (\del_{\F} + \adel_{\F}) \circ \mathbf{i}\Lambda_{\F} \circ (\adel_{\F} - \del_{\F}).
 \end{align*}
Another application of the Nakano identities yields 
 \begin{align*}
\left(- \del_{\F} \circ \adel_{\F} + \adel_{\F} \circ \del_{\F} \right) \circ \mathbf{i} \Lambda_{\F} + (\del_{\F} + \adel_{\F}) \circ (\mathbf{i} \adel_{\F} \circ \Lambda_{\F} + \del_{\F}^{\dagger} - \mathbf{i} \del_{\F} \circ \Lambda_{\F} + \adel_{\F}^{\dagger} ).
 \end{align*}
Cancelling the obvious terms, we finally arrive at the expression
\begin{align*}
\nabla \circ \nabla^{\dagger} = \del_{\F} \circ \del^{\dagger}_{\F} + \del_{\F} \circ \adel_{\F}^{\dagger} + \adel_{\F} \circ \del_{\F}^{\dagger} + \adel_{\F} \circ \adel_{\F}^{\dagger}. 
\end{align*}
An analogous calculation for $\nabla^{\dagger} \circ \nabla$ yields the identity 
\begin{align*}
\nabla^{\dagger} \circ \nabla = \del_{\F}^{\dagger} \circ \del_{\F} + \del_{\F}^{\dagger} \circ \adel_{\F} + \adel^{\dagger}_{\F} \circ \del_{\F} + \adel_{\F}^{\dagger} \circ \adel_{\F}. 
\end{align*} 
The identities in \eqref{eqn:anticommrels} now imply that 
\begin{align*}
\Delta_{\nabla} =   \del_{\F} \circ \del_{\F}^{\dagger} +  \del_{\F}^{\dagger} \circ \del_{\F} +  \adel_{\F} \circ \adel_{\F}^{\dagger} + \adel_{\F}^{\dagger} \circ \adel_{\F} 
=  \Delta_{\del} + \Delta_{\adel}.
\end{align*}
Finally, the other identities in \eqref{eqn:ChernAkizukiNak} can now be concluded  from the Akizuki--Nakano identity  \eqref{eqn:AkizNak}.
 \end{proof}

\subsection{Spectral Gaps}

In this section we use the identity established in the previous section, along with the Akizuki--Nakano identity to produce spectral gaps for twisted Laplace and Dirac operators. We begin with a simple lemma, giving an explicit value for the curvature commutator on homogeneous forms. Just as in the classical case, is it a consequence of the $\frak{sl}_2$-representation presented in \textsection \ref{subsection:sl2}.

\begin{lem} \label{lem:PosNegChernLefAction}
Let $\mathcal{F}_{+}$, and $\mathcal{F}_-$,  be positive, and respectively negative, vector bundles  over a $2n$-dimensional  K\"ahler structure  $(\Omega^{(\bullet,\bullet)},\kappa)$, with closed integral. It holds that  
\begin{align*}
[\mathbf{i}\nabla^2,\Lambda_{\mathcal{F}_{\pm}}](\omega \otimes f) = \pm \theta_{\pm}(k-n)(\omega \otimes f), & & \textrm{ for all ~} \omega \otimes f \in \Omega^k \!\otimes_B \mathcal{F}_{\pm},
\end{align*}
where $\nabla^2(f) = \mp \theta_{\pm}  \kappa \otimes f$, with $\theta_{\pm} \in \mathbb{R}_{>0}$.
\end{lem}
\begin{proof} 
For $\mathcal{F}_+$ the claimed identity follows from 
\begin{align*}
[\mathbf{i}\nabla^2,\Lambda_{\mathcal{F}_+}](\omega \otimes f) = & \,\, \mathbf{i} \nabla^2 \circ \Lambda_{\mathcal{F}_+}(\omega \otimes f) - \mathbf{i} \Lambda_{\mathcal{F}_+} \circ \nabla^2(\omega \otimes f)\\
                                                                                          = & \,\, \mathbf{i}  \Lambda_{\kappa}(\omega) \wedge \nabla^2(f) - \mathbf{i} \Lambda_{\mathcal{F}_+}(\omega \wedge \nabla^2(f)) \\
                                                                                          = & \,\, \theta (\Lambda_{\kappa}(\omega) \wedge \kappa) \otimes f -  \theta \Lambda_{\mathcal{F}_+}(\omega \wedge \kappa \otimes f) \\
                                                                                          = & \,\,  \theta (L_{\kappa} \circ \Lambda_{\kappa}(\omega)) \otimes f - \theta (\Lambda_{\kappa} \circ L_{\kappa}(\omega)) \otimes f\\
                                                                                          = & \,\, \theta ([L_{\kappa},\Lambda_{\kappa}](\omega)) \otimes f\\
                                                                                          = & \,\,  \theta (k-n)\,\omega \otimes f.
\end{align*}
The case of $\mathcal{F}_-$ is completely analogous, amounting to a change of sign.
\end{proof}

We are now ready to conclude some spectral properties of twisted Laplace and Dolbeault--Dirac operators from the behaviour of the curvature of their Chern connection. Here we note that the domain of the operators is the lower half of the Hodge diamond.

\begin{lem} \label{lem:lowerbound}
Let $(\Omega^{(\bullet,\bullet)},\kappa)$ be a noncommutative K\"ahler structure  with closed integral. Then a lower bound for the eigenvalues of the operators 
\begin{align} 
\Delta_{\adel_{\F}}: \bigoplus_{k=1}^{n-1} \Omega^{k} \to \bigoplus_{k=1}^{n-1} \Omega^{k}, & & \, \Delta_{\nabla}: \bigoplus_{k=1}^{n-1} \Omega^{k} \to \bigoplus_{k=1}^{n-1} \Omega^{k},
\end{align}
is given by $\theta$, and $3\theta$, respectively.
\end{lem}
\begin{proof}
Since the Laplacians $\Delta_{\del_{\F}}, \, \Delta_{\adel_{\F}}$ are degree zero operators, their space of eigenvectors will be spanned by forms that are homogeneous with respect to the $\mathbb{Z}_{\geq 0}$-grading of the calculus. Moreover, since the commutator  $[\mathbf{i}\nabla^2,\Lambda_{\mathcal{F}_{\pm}}]$ acts on each homogeneous subspace $\Omega^k$ as scalar multiplication, the Akizuki--Nakano identity implies that the eigenspaces of $\Delta_{\del_{\F}}$ and $\Delta_{\adel_{\F}}$ coincide. Now the fact that $\Delta_{\del_{\F}}$ is a positive operator implies that its eigenvalues are positive. Thus it follows from Lemma \ref{lem:PosNegChernLefAction} above that any eigenvector of $\Delta_{\adel_{\F}}$ has eigenvalue greater than or equal to $\theta$. 

It follows from \eqref{eqn:ChernAkizukiNak} that the eigenspaces of $\Delta_{\nabla}$ and $\Delta_{\adel_{\F}}$ coincide. Moreover the eigenvalues of $\Delta_{\nabla}$ will be greater than or equal to the eigenvalues of  $\Delta_{\adel_{\F}} + [\Lambda_{\F},\mathbf{i}\nabla^2]$, and so, greater than $3\theta$.
\end{proof}

Finally, we come to the spectral gap for the Dolbeault--Dirac and Laplace operators acting on the on its anti-holomorphic subcomplex of the twisted Hodge diamond. In the classical case this includes the spin Dirac operator of a K\"ahler manifold, as discussed in \textsection \ref{subsection:QFMSpin}.

\begin{thm}\label{thm:lowerbound}
If the Dolbeault--Dirac operator $\Delta_{\adel_{\F}}$ is diagonalisable, then $\theta$ is a lower bound for the eigenvalues of the operator
\begin{align} 
\Delta_{\adel_{\F}}:  \Omega^{(0,\bullet)} \to  \Omega^{(0,\bullet)}.
\end{align}
Moreover,  $\sqrt{\theta}$ is a lower bound for the absolute value of the eigenvalues of the operator
\begin{align} 
D_{\adel_{\F}}: \Omega^{(0,\bullet)} \to  \Omega^{(0,\bullet)}.
\end{align}
\end{thm}
\begin{proof}
If we now assume that the Dolbeault--Dirac operator $D_{\adel_{\F}}$ is diagonalisable, then, as established in \cite[Theorem 6.4]{OSV}, we have a noncommutative generalisation of Hodge decomposition. In particular, for the case of $(0,n)$-forms, we have the decomposition
\begin{align*}
\Om^{(0,n)} \otimes_B {\F}  = \mathcal{H}^{(0,n)}_{\ol{\del}_{\F}} \oplus  \ol{\del}_\F\!\left(\Om^{(0,n-1)}\right)\!,
\end{align*}
where $ \mathcal{H}^{(0,n)}_{\ol{\del}_{\F}}$ are the $\F$-valued harmonic $(0,n)$-forms, that is to say, the $(0,n)$-forms contained in the kernel of the Laplacian $\Delta_{\adel_{\F}}$. Now each summand in the decomposition is closed under the action of the Laplacian (see for example \cite[Lemma 3.1]{DOS1}) hence the Laplacian is diagonalisable on each summand, and in particular on $\ol{\del}_\F(\Om^{(0,n-1)} \otimes_B {\F})$. Now for any eigenvector $\alpha \in  \ol{\del}_\F(\Om^{(0,n-1} \otimes_B {\F})$, the form $\adel_{\F}(\alpha)$ is again an eigenvector. Indeed, denoting the eigenvalue of $\alpha$ by $\lambda$, we see that 
\begin{align*}
\Delta_{\adel_{\F}}(\adel_{\F}(\omega)) = \adel_{\F}(\adel^{\dagger}_{\F} \circ \adel_{\F}(\omega)) = \lambda  \adel_{\F}(\omega).
\end{align*}
Recalling that the eigenvectors in $\Omega^{(0,n-1)} \otimes_B \F$ are bounded below by $\theta$, we now see that the non-zero eigenvalues of the eigenvectors in $\Omega^{(0,n)} \otimes_B \F$ are also bounded below by $\theta$. Finally, we note that since the Laplace operator is by definition the square of the Dolbeault--Dirac operator, it is clear that the absolute value of its eigenvalues are bounded below by $\sqrt{\theta}$.
\end{proof}

\begin{cor}
If  $\Omega^{(0,n-1)} \otimes \F$ admits an anti-holomorphic form, that is to say, if there exists  a non-zero 
$$
\omega \otimes f \in \ker\!\left(\del_{\F}: \Omega^{(0,n-1)} \otimes_B \F \to  \Omega^{(1,n)} \otimes_B \F\right)\!,
$$
then $\theta$, and $\sqrt{\theta}$, are eigenvalues of $\Delta_{\adel_{\F}}$, and $D_{\adel_{\F}}$,  respectively. In particular, the lower bounds are strict.
\end{cor}
\begin{proof}
On any such  $\omega \otimes f$ the Akizuki--Nakano identity identity reduces to
\begin{align*}
\Delta_{\adel_{\F}}(\omega \otimes f) = [\mathbf{i}\nabla^2, \Lambda_\F](\omega \otimes f) = \theta \, \omega \otimes f.
\end{align*}
Thus $\theta$ is an eigenvalue of $\Delta_{D_{\adel}}$, implying that $\sqrt{\theta}$ is an eigenvalue of $\Delta_{\adel_{\F}}$.
\end{proof}

\section{Covariant Hermitian Modules} \label{section:CovHermMods}

In this section we restrict to the special case of quantum homogeneous spaces and covariant Hermitian modules. We characterise covariant Hermitian modules in terms of covariant inner products on the associated inducing comodule. Moreover, we show that an Hermitian structure on a covariant complex structure equips it with the structure of a covariant Hermitian module. These results help  to elucidate the processes at work in the next section, where we treat our motivating examples the irreducible quantum flag manifolds.

\subsection{Preliminaries on Covariant Differential Calculi and Hermitian Structures} \label{subsection:covariantHS}

A {\em covariant Hermitian} structure for $\Om^\bullet$  is an Hermitian structure $(\Om^{(\bullet,\bullet)},\s)$ such that $\Om^{(\bullet,\bullet)}$ is a covariant complex structure, and the Hermitian form $\sigma$ is left $A$-coinvariant, which is to say $\DEL_L(\s) = 1 \oby \s$. A {\em covariant K\"ahler} structure is a covariant Hermitian structure which is also a K\"ahler structure.  Note that in the covariant case, in addition to being $P$-bimodule maps, $L_{\sigma}$, $\ast_\s$, and $\Lambda_{\sigma}$ are also left $A$-comodule maps. 

Let us now consider the special case of a quantum homogeneous $A$-space $B \subseteq A$, and recall from Appendix \ref{app:TAK} the category ${}^A_B\mathrm{mod}_0$ of relative Hopf modules associated to $B$. An \emph{Hermitian relative Hopf module} is an Hermitian module $(\F, h_\F)$ such that $\F$ is an object in ${}^A_B\mathrm{mod}_0$ and the isomorphism 
$
h_{\F}: \overline{\F} \to \, {}^{\vee}\!\F
$
is a morphism in $^A_B\mathrm{mod}_0$, where the conjugate $\overline{\F}$ and dual ${}^{\vee}\!\F$ modules are understood as objects in $^A_B\mathrm{mod}_0$ in the sense of Appendix \ref{app:TAK}. For any $V \in {}^{\pi_B}\mathrm{mod}_0$, its dual and conjugate modules are always isomorphic (for details see \cite[Theorem~11.27]{KSLeabh}) implying the existence of covariant Hermitian structures as discussed in \textsection \ref{section:CovHermMods}. Moreover, for any simple object $\F$, its conjugate $\overline{\F}$ and its dual ${}^{\vee}\!\F$ will again be simple, implying that  any covariant Hermitian structure will be unique up to positive scalar multiple.

Consider next a covariant differential calculus $\Omega^{\bullet}$ over $B$, endowed with a covariant complex structure $\Omega^{(\bullet,\bullet)}$.  A \emph{ holomorphic relative Hopf module} is a holomorphic module $(\F,\adel_{\F})$ over $B$, such that $\F$ is an object in ${}^A_B\mathrm{mod}_0$ and $\adel_{\F}:\F \to \Omega^{(0,1)} \otimes_B \F$ is a left $A$-comodule map. An Hermitian holomorphic module $(\F, h_\F, \del_F)$ is said to be covariant if its constituent Hermitian and holomorphic modules are covariant. In this case, the Chern connection is always a left $A$-comodule map, see \cite[\textsection 7.1]{OSV}. 

\subsection{A Characterisation of Covariant Hermitian Modules}

In this section we use Takeuchi's equivalence to give an alternative characterisation of covariant Hermitian structures. This description generalises the classical fiberwise definition of an Hermitian metric on a complex manifold.

\begin{lem} \label{lem:TAKGU}
Let $B \subseteq A$ be a quantum homogeneous space, $\F \in {}^A_B\mathrm{mod}$ a relative Hopf module, and $h:\overline{\F} \to {}^{\vee}\F$ a morphism. Then a commutative diagram is given by 
\begin{align*}
\xymatrix{ 
~~~~~~~~~~~~~ \F \times \F  ~~       \ar[drr]_{g_{\F}}           \ar[rrrr]^{\unit \times \unit~~~~~~~~~~~~~~~}     &  & & &    ~~   A \square_{\pi_B} \Phi(\F) \times A \square_{\pi_B} \Phi(\F)  ~~~~~~~~~~~~~~~~~~~~~~~~~~~     \ar[dll]^{g_{\unit}}  \\
                  &    &  B  & &\\
}
\end{align*}
where 
$
g_{\unit}
$ 
is the map uniquely defined by 
$$
g_{\unit}\!\left(\sum_i a_i \otimes [f_i], \sum_j a_j \otimes [f_j]\right) := \sum_{i,j} a_ia_j^* \e(g_{\F}(f_i,f_j)).
$$
\end{lem}
\begin{proof}
It follows from the definition of $g_{\unit}$ that 
$$
\e(g_{\F}(af_1,bf')) = \e(ag_{\F}(f,f')b^*) = \e(a)\e(g_{\F}(f,f'))\overline{\e(b)},
$$
implying that $g_{\unit}$ is well-defined with respect to the quotient $\Phi(\F) = \F/B^+\F$. The calculation 
\begin{align*}
g_{\unit}(\unit(f),\unit(f'))  = & \, g_{\unit}(f_{(-1)} \otimes [f_{(0)}], f'_{(-1)} \otimes [f'_{(0)}]) \\ 
= & \,  f_{(-1)}(f'_{(-1)})^* \e(g_{\F}(f_{(0)},f'_{(0)})) \\
= & \, (\id \otimes \e) \circ \Delta_L(g_{\F}(f,f'))\\
= & \, g_{\F}(f,f')_{(1)} \e(g_{\F}(f,f')_{(2)})\\
 =  & \, g_{\F}(f,f'),
\end{align*}
implies that the diagram is commutative.
\end{proof}

For $V$ an object in ${}^H\mathrm{mod}$, an inner product is said to be \emph{covariant} if the associated vector space space isomorphism $\overline{V} \simeq {}^{\vee}V$, between the conjugate comodule of $V$ and its dual,  is a left $H$-comodule map. As we now see, covariant inner products are equivalent to covariant Hermitian structures.

\begin{prop} \label{prop:covHermStructures}
A morphism $h:\overline{\F} \to {}^{\vee}\F$ is an Hermitian structure if and only if an  inner product on $\Phi(\F)$  is given by 
\begin{align} \label{eqn:innerproduct}
  \Phi(\F) \times \Phi(\F)  \to \mathbb{C}, & &  [f] \otimes [g] \mapsto \e(h(f,g)).
\end{align}
This gives a bijection between covariant Hermitian structures $h:\F \to \mathcal{F}$ and covariant inner products on $\Phi(\F)$.
\end{prop}
\begin{proof}
Let us assume that $h$ is an Hermitian structure. Applying the functor $\Phi$ we get an isomorphism $\Phi(h):\Phi(\overline{\F}) \to \Phi({}^{\vee}\F)$, and hence \eqref{eqn:innerproduct} is a non-degererate bilinear form. Moreover, since $g_{\F}$ is by assumption symmetric, the bilinear form is also symmetric. Moreover, positivity of $g_{\F}$ implies that 
$$
([f],[f]) = g_{\F}(f,f) > 0,
$$
meaning that \eqref{eqn:innerproduct} is an inner product. 

In the opposite direction, consider a covariant inner product with associated isomorphism $\overline{V} \simeq {}^{\vee}V$. Denoting  the corresponding sesquilinear form $g_{\unit}:\Psi(V) \times \Psi(V) \to B$, and choosing   an orthonormal basis $\{[f_i]\}_i$ of $V$, then Lemma \ref{lem:TAKGU} implies that
$$
g_{\unit}\!\left(\sum_i a_i \otimes [f_i], \sum_j a_j \otimes [f_j]\right) := \sum_{i} a_ia_j^*.
$$
Thus we see that $g_{\unit}$ is positive definite and conjugate symmetric. The claimed bijective correspondence is now clear.
\end{proof}

In general, the metric associated to a positive definite Hermitian structure does not necessarily induce an isomorphism between the conjugate of the underlying calculus and its dual. Thus an Hermitian structure does not necessarily give a calculus the structure of an Hermitian module. However, as the following corollary shows, for covariant Hermitian structures over quantum homogeneous spaces this problem does not arise.

\begin{cor}
Let  $B \subseteq A$ be a quantum homogeneous space and $\Omega^{\bullet} \in {}^A_B\mathrm{mod}_0$ a covariant differential calculus over $B$. For any positive definite covariant Hermitian structure  $(B, \Omega^{(\bullet,\bullet)})$, an Hermitian module is given by the pair $(\Omega^{\bullet},h_{\sigma})$, where 
\begin{align*}
h_{\sigma}: \overline{\Omega^{\bullet}} \to {}^{\vee}(\Omega^{\bullet}), & & \omega^* \mapsto g_{\kappa}(\omega, \, - \,) = \ast_{\sigma}(\ast_{\sigma}(\omega^*) \wedge \, - \,).
\end{align*}
\end{cor}
\begin{proof}
The assumption of positivity of the Hermitian structure implies that $g_{\kappa}$ induces an inner product on $\Phi(\Omega^{\bullet})$. Thus by Proposition \ref{prop:covHermStructures} we have an Hermitian structure as claimed. 
\end{proof}


\section{The Irreducible Quantum Flag Manifolds as CQH-K\"ahler Spaces} \label{section:DJ}

In this section we treat motivating set of examples, the irreducible quantum flag manifolds $\O_q(G/L_S)$ endowed with their Hekenberger--Kolb differential calculi.  We begin by recalling some necessary preliminaries on Drinfeld--Jimbo quantum groups, quantum flag manifolds, and K\"ahler structures for the Heckenberger--Kolb calculi. We then establish positivity of the associated metric, and conclude from it a spectral gap for Dolbeault--Dirac and Laplace operators twisted by a negative line bundle. Moreover, we use the noncommutative hard Lefschetz theorem to conclude nonvanishing of even degree cohomologies. 

\subsection{Drinfeld--Jimbo Quantum Groups}

Let $\frak{g}$ be a finite-dimensional complex simple Lie algebra of rank $r$. For $q \in \bR$ such that  $q \neq -1,0,1$, we denote by $U_q(\frak{g})$ the Drinfeld--Jimbo quantised enveloping algebra. We denote the generators by $E_i,F_i,K_i$, for $i = 1, \dots, r$ and follow the conventions of \cite[\textsection 7]{KSLeabh}.  Moreover, we endow $U_q(\frak{g})$ with the compact real form Hopf $*$-algebra.

We fix a Cartan subalgebra and  a set of positive roots $\Pi = \{\alpha_1, \dots, \alpha_r\}$. We denote the coroots by $\alpha_i^{\vee}$, and the dual basis of  fundamental weights  by $\{\varpi_1, \dots, \varpi_r\}$. 
We denote by ${\mathcal P}$, and $\mathcal{P}^+$, the integral weight lattice of $\frak{g}$, and the cone of  {\em dominant integral weights}, respectively.  For each $\mu\in\mathcal{P}^+$, we denote by $V_{\mu}$ the corresponding finite-dimensional type-$1$ $U_q(\frak{g})$ highest weight module  $V_\mu$. We denote by $_{U_q(\frak{g})}\mathbf{type}_1$ the full subcategory of ${U_q(\frak{g})}$-modules whose objects are finite  sums of type-1 modules.  Note  that $\,_{U_q(\frak{g})}\mathbf{type}_1$ is abelian, semisimple, and  equivalent to the category of  finite-dimensional representations of $\frak{g}$. Moreover, $\,_{U_q(\frak{g})}\mathbf{type}_1$ admits the structure of a braided monoidal category (coming from the $h$-adic quasi-triangular structure of the Drinfeld--Jimbo algebras). Given a choice of bases $\{e_i\}_{i=1}^{\mathrm{dim}(V)}$, and $\{f_i\}_{i=1}^{\mathrm{dim}(W)}$, for two finite-dimensional $U_q(\frak{g})$-modules $V$ and $W$, its associated {\em $R$-matrix} $R^{ij}_{kl}$ is defined by
\begin{align*}
\widehat{R}_{V,W}(e_i \otimes f_j) = \sum_{k,l} (\widehat{R}_{V,W})^{kl}_{ij} f_k \otimes e_l.
\end{align*}
It follows from Lusztig and Kashiwara's theory of crystal bases \cite{LusztisLeabh,KashQCrtstal} that one can choose a weight basis for any $U_q(\frak{g})$-module such that the associated $R$-matrix coefficients $\widehat{R}_{V,W}$ are Laurent polynomials in $q$. (See \cite{MTSAlexCh}, and references therein, for a more detailed discussion.) We call such a basis a {\em Laurent basis} of $V$. The existence of Laurent bases will be used below to establish a positive definiteness result for the K\"ahler structures considered in this section.

\subsection{Quantum Coordinate Algebras and the Quantum Flag Manifolds}

Let $V$ be a finite-dimensional $U_q(\frak{g})$-module, $v \in V$, and $f \in V^*$, the linear dual of $V$. Consider the function  $c^{\textrm{\tiny $V$}}_{f,v}:U_q(\frak{g}) \to \bC$ defined by $c^{\textrm{\tiny $V$}}_{f,v}(X) := f\big(X(v)\big)$. The {\em coordinate ring} of $V$ is the subspace
\begin{align*}
C(V) := \text{span}_{\mathbb{C}}\!\left\{ c^{\textrm{\tiny $V$}}_{f,v} \,| \, v \in V, \, f \in V^*\right\} \sseq U_q(\frak{g})^*.
\end{align*}
We note that $C(V)$ is contained in $U_q(\frak{g})^\circ$, the Hopf dual of $U_q(\frak{g})$, and moreover that a Hopf subalgebra of $U_q(\frak{g})^\circ$ is given by 
\begin{align*}
\O_q(G) := \bigoplus_{\mu \in \mathcal{P}^+} C(V_{\mu}).
\end{align*}
We call $\O_q(G)$ the {\em quantum coordinate algebra of $G$}, where $G$ is the compact, connected, simply-connected, simple Lie group  having $\frak{g}$ as its complexified Lie algebra.
The Hopf algebra $\O_q(G)$ is cosemisimple by construction. The compact real form of $U_q(\frak{g})$ dualises to a $*$-structure on $\O_q(G)$ with respect to which the Haar functional is positive.

For $S$ a subset of simple roots,  consider the Hopf $*$-subalgebra of $U_q(\frak{g})$ given by 
\begin{align*}
U_q(\frak{l}_S) := \big< K_i, E_j, F_j \,|\, i = 1, \ldots, r; j \in S \big>.
\end{align*} 
Just as for $U_q(\frak{g})$, each finite-dimensional $U_q(\frak{l}_S)$-module
decomposes into a direct sum weight spaces.

The {\em quantum flag manifold associated to} $S$ is the space of invariants
\begin{align*}
\O_q\big(G/L_S\big) := {}^{U_q(\frak{l}_S)}\O_q(G).
\end{align*} 
The algebra $\O_q(G)$ is faithfully flat as a right $\O_q(G/L_S)$-module and in particular is a quantum homogeneous space  (see \cite{PPACROB} for details).

Restricting to the case $S$ contains a single element $\alpha_x$, and choose for $V_{\alpha_x}$ a weight basis $\{v_i\}_{i}$, with corresponding dual basis $\{f_i\}_{i}$ for the dual module $V_{\alpha_x}^{\vee}$
As shown in \cite[Proposition 3.2]{HK}, a set of generators for $\O_q(G/L_S)$ is given by 
\begin{align*}
z_{ij} := c^{V_{\alpha_x}}_{f_i,v_N}c^{V_{\alpha_x}^{\vee}}_{v_j,f_N}  & & \text{ for } i,j = 1, \dots, N := \dim(V_{\alpha_x}),
\end{align*}
where $v_N$, and $f_N$, are the highest weight basis elements of $V_{\alpha_x}$, and $V_{\alpha_x}^{\vee}$, respectively.

\subsection{The Heckenberger--Kolb Calculi} 

If $S = \{\alpha_1, \ldots, \alpha_r\}\bs \{\a_i\}$, where $\alpha_i$ has coefficient $1$ in the expansion of the highest root of $\frak{g}$, then we say that the associated quantum flag manifold is  {\em irreducible}. In the classical limit of $q=1$, these homogeneous spaces reduce to the family of compact Hermitian symmetric spaces \cite{BastonEastwood}. 
Table $1$ below gives a useful diagrammatic presentation of the set of simple roots defining the irreducible quantum flag manifolds.

\medskip

\begin{center} 
\captionof{table}{Irreducible Quantum Flag Manifolds: organised by series,  with defining  crossed  node  with respect to the standard numbering  of simple roots  \cite[\textsection 11.4]{Humph}, CQGA homogeneous space symbol and name, as well as the complex dimension $M$ of the corresponding classical complex manifold} \label{table:CQFMs}
{\tiny
\begin{tabular}{|c|c|c|c| }

\hline

& & &  \\

\small $A_n$&
\begin{tikzpicture}[scale=.5]
\draw
(0,0) circle [radius=.25] 
(8,0) circle [radius=.25] 
(2,0)  circle [radius=.25]  
(6,0) circle [radius=.25] ; 

\draw[fill=black]
(4,0) circle  [radius=.25] ;

\draw[thick,dotted]
(2.25,0) -- (3.75,0)
(4.25,0) -- (5.75,0);

\draw[thick]
(.25,0) -- (1.75,0)
(6.25,0) -- (7.75,0);
\end{tikzpicture} & \small $\O_q(\text{Gr}_{k,n+1})$ & \small quantum Grassmannian \\

 &  & &   \\

\small $B_n$ &
\begin{tikzpicture}[scale=.5]
\draw
(4,0) circle [radius=.25] 
(2,0) circle [radius=.25] 
(6,0)  circle [radius=.25]  
(8,0) circle [radius=.25] ; 
\draw[fill=black]
(0,0) circle [radius=.25];

\draw[thick]
(.25,0) -- (1.75,0);

\draw[thick,dotted]
(2.25,0) -- (3.75,0)
(4.25,0) -- (5.75,0);

\draw[thick] 
(6.25,-.06) --++ (1.5,0)
(6.25,+.06) --++ (1.5,0);                      

\draw[thick]
(7,0.15) --++ (-60:.2)
(7,-.15) --++ (60:.2);
\end{tikzpicture} & \small $\O_q(\mathbf{Q}_{2n+1})$ & \small $\begin{array} {cc} \mathrm{odd} ~\mathrm{quantum} \\ \mathrm{quadric} \end{array}$  \\

 & & &     \\

\small $C_n$& 
\begin{tikzpicture}[scale=.5]
\draw
(0,0) circle [radius=.25] 
(2,0) circle [radius=.25] 
(4,0)  circle [radius=.25]  
(6,0) circle [radius=.25] ; 
\draw[fill=black]
(8,0) circle [radius=.25];

\draw[thick]
(.25,0) -- (1.75,0);

\draw[thick,dotted]
(2.25,0) -- (3.75,0)
(4.25,0) -- (5.75,0);

\draw[thick] 
(6.25,-.06) --++ (1.5,0)
(6.25,+.06) --++ (1.5,0);                      

\draw[thick]
(7,0) --++ (60:.2)
(7,0) --++ (-60:.2);
\end{tikzpicture} &\small   $\O_q(\mathbf{L}_{n})$ & \small 
quantum Lagrangian   \\

  &  & & \small Grassmannian    \\ 

\small $D_n$& 
\begin{tikzpicture}[scale=.5]

\draw[fill=black]
(0,0) circle [radius=.25] ;

\draw
(2,0) circle [radius=.25] 
(4,0)  circle [radius=.25]  
(6,.5) circle [radius=.25] 
(6,-.5) circle [radius=.25];

\draw[thick]
(.25,0) -- (1.75,0)
(4.25,0.1) -- (5.75,.5)
(4.25,-0.1) -- (5.75,-.5);

\draw[thick,dotted]
(2.25,0) -- (3.75,0);
\end{tikzpicture} &\small   $\O_q(\mathbf{Q}_{2n})$ & \small  $\begin{array} {cc} \mathrm{even} ~\mathrm{quantum} \\ \mathrm{quadric} \end{array}$ \\ 
 &  & &    \\ 

\small $D_n$ & 
\begin{tikzpicture}[scale=.5]
\draw
(0,0) circle [radius=.25] 
(2,0) circle [radius=.25] 
(4,0)  circle [radius=.25] ;

\draw[fill=black] 
(6,.5) circle [radius=.25] 
(6,-.5) circle [radius=.25];

\draw[thick]
(.25,0) -- (1.75,0)
(4.25,0.1) -- (5.75,.5)
(4.25,-0.1) -- (5.75,-.5);

\draw[thick,dotted]
(2.25,0) -- (3.75,0);
\end{tikzpicture} &\small   $\O_q(\textbf{S}_{n})$ & \small  $\begin{array} {cc} \mathrm{quantum} ~\mathrm{spinor} \\ \mathrm{variety} \end{array}$  \\
 &  & &    \\ 

\small $E_6$& \begin{tikzpicture}[scale=.5]
\draw
(2,0) circle [radius=.25] 
(4,0) circle [radius=.25] 
(4,1) circle [radius=.25]
(6,0)  circle [radius=.25] ;

\draw[fill=black] 
(0,0) circle [radius=.25] 
(8,0) circle [radius=.25];

\draw[thick]
(.25,0) -- (1.75,0)
(2.25,0) -- (3.75,0)
(4.25,0) -- (5.75,0)
(6.25,0) -- (7.75,0)
(4,.25) -- (4, .75);
\end{tikzpicture}

 &\small  $\O_q(\mathbb{OP}^2)$ & \small  $\begin{array} {cc} \mathrm{quantum} ~\mathrm{Caley} \\ \mathrm{plane} \end{array}$  \\
 &   & &    \\ 
\small $E_7$& 
\begin{tikzpicture}[scale=.5]
\draw
(0,0) circle [radius=.25] 
(2,0) circle [radius=.25] 
(4,0) circle [radius=.25] 
(4,1) circle [radius=.25]
(6,0)  circle [radius=.25] 
(8,0) circle [radius=.25];

\draw[fill=black] 
(10,0) circle [radius=.25];

\draw[thick]
(.25,0) -- (1.75,0)
(2.25,0) -- (3.75,0)
(4.25,0) -- (5.75,0)
(6.25,0) -- (7.75,0)
(8.25, 0) -- (9.75,0)
(4,.25) -- (4, .75);
\end{tikzpicture} &\small   $\O_q(\textrm{F})$  
& \small  $\begin{array} {cc} \mathrm{quantum} ~\mathrm{Freudenthal} \\ \mathrm{variety} \end{array}$  \\

& & &    \\  

\hline
\end{tabular}
}
\end{center}

\medskip

The irreducible quantum flag manifolds are distinguished by the existence of an essentially unique $q$-deformation of their classical de Rham complex. 
Explicitly, over any irreducible quantum flag manifold $\O_q(G/L_S)$, there exists a unique finite-dimensional left $\O_q(G)$-covariant differential calculus
\[
\Omega^{\bullet}_q(G/L_S) \in {}^{~~~~\O_q(G)}_{\O_q(G/L_S)}\mathrm{mod}_0
\]
of \emph{classical dimension}, that is to say, satisfying
\begin{align*}
\dim \Phi\!\left(\Omega^{k}_q(G/L_S)\right) = \binom{2M}{k}, & & \text{ for all \,} k = 0, \dots, 2 M,
\end{align*}
where $M$ is the complex dimension of the corresponding classical manifold (see Table \ref{table:2} below for explicit values). We call $\Omega^{\bullet}_q(G/L_S)$ the \emph{Heckenberger--Kolb calculus} of the quantum flag manifold $\O_q(G/L_S)$.

\begin{eg}
We now consider the special case of \emph{quantum projective space} $\O_q(\mathbb{CP}^{n})$,  the simplest type of quantum Grassmannian. Explicitly, it is the $A_n$-type  irreducible quantum flag manifold corresponding to the first or  last crossed node of the Dynkin diagram, which is to say, the nodes 

\begin{tabular}{ccccc}
~~~~~~
&
\begin{tikzpicture}[scale=.5]
\draw
(2,0) circle [radius=.25] 
(4,0) circle [radius=.25] 
(6,0)  circle [radius=.25]  
(8,0) circle [radius=.25] ; 
\draw[fill=black]
(0,0) circle [radius=.25];

\draw[thick]
(.25,0) -- (1.75,0);

\draw[thick,dotted]
(2.25,0) -- (3.75,0)
(4.25,0) -- (5.75,0);

\draw[thick]
(6.25,0) -- (7.75,0);

\end{tikzpicture} 
&
&
~~~~~~~~
or
~~~~~~~~~
&
\begin{tikzpicture}[scale=.5]
\draw
(2,0) circle [radius=.25] 
(4,0) circle [radius=.25] 
(6,0)  circle [radius=.25]  
(0,0) circle [radius=.25];
\draw[fill=black]
(8,0) circle [radius=.25] ; 

\draw[thick]
(.25,0) -- (1.75,0);

\draw[thick,dotted]
(2.25,0) -- (3.75,0)
(4.25,0) -- (5.75,0);

\draw[thick]
(6.25,0) -- (7.75,0);

\end{tikzpicture}. 
\end{tabular}

The \emph{quantum projective line} $\O_q(\mathbb{CP}^1)$, the simplest example of a quantum flag manifold, is usually denoted by $\O_q(S^2)$. As it was originally introduced by Podle\'s \cite{PodlesSphere}, it  is usually called the  \emph{Podle\'s sphere}.  For this special case, the Heckenberger--Kolb calculus reduces to the calculus $\Omega_q^{\bullet}(S^2)$ originally introduced by Podle\'s in \cite{PodlesCalc} and usually known as the {\em Podle\'s calculus}.
\end{eg}

\subsection{Generators and Relations for the Differential Calculus $\Omega^{\bullet}_q(G/L_S)$} \label{subsection:gensandrels}

For each irreducible quantum flag manifold, the defining relations of the maximal prolongation $\Omega^{\bullet}_q(G/L_S)$ are a subtle and intricate $q$-deformation of the classical Grassmann anti-commutative relations. (For example,  an explicit presentation of the relations of the  Podle\'s calculus $\Omega^{\bullet}_q(S^2)$ is given in \cite{PodlesCalc}.) In an impressive technical achievement, a complete $R$-matrix description of the general $\Omega_q^{\bullet}(G/L_S)$ relations was given  in \cite[\textsection 3.3]{HKdR}. We recall here this presentation, following the original conventions of Heckenberger and Kolb.
In particular, we use the following $R$-matrix notations, defined with respect to the index set $J := \{1,\dots, \dim(V_{\varpi_s})\}$:
$$
\begin{array} {l}
\widehat{R}_{V_{\varpi_s},V_{\varpi_s}}(v_i \otimes v_j)   =: \sum_{k,l \in J} \widehat{R}^{kl}_{ij}\, v_k \otimes v_l, \\
\\
  \widehat{R}_{V_{-w_0(\varpi_s)},V_{\varpi_s}}(f_i \otimes v_j)   =: \sum_{k,l \in J} \acute{R}^{-kl}_{~~ij}\, v_k \otimes f_l,  \\
  \\
   \widehat{R}_{V_{\varpi_s},V_{-w_0(\varpi_s)}}(v_i \otimes f_j)   =: \sum_{k,l \in J} \grave{R}^{-kl}_{~~ij}\, f_k \otimes v_l,
   \\
   \\
     \widehat{R}_{V_{-w_0(\varpi_s)},V_{-w_0(\varpi_s)}}(f_i \otimes f_j)   =: \sum_{k,l \in J} \widecheck{R}^{kl}_{~~ij}\, f_k \otimes f_l.    
\end{array}
$$
We denote by $\widehat{R}^-, \acute{R}, \grave{R}$, and $\widecheck{R}^-$, the inverse matrices of $\widehat{R}, \acute{R}^-, \grave{R}^-$, and $\widecheck{R}$ respectively. The calculus $\Om^{\bullet}_q(G/L_S)$ can be described as the tensor algebra of the $\O_q(G/L_S)$-bimodule $\Omega^1_q(G/L_S)$ subject to three sets of matrix relations, given in terms of the coordinate matrix $\mathbf{z} := (z_{ij})_{(ij)}$. First are the {\em holomorphic relations}
\begin{align} \label{eqn:holoCALCRELS}
\widehat{Q}_{12} \acute{R}_{23} \del \mathbf{z} \wedge \del \mathbf{z} = 0, & & \widecheck{P}_{34} \acute{R}_{23} \del \mathbf{z} \wedge \del \mathbf{z} = 0,
\end{align}
where  we have used leg notation,  and have denoted
\begin{align*}
\widehat{Q} := \widehat{R} +q^{(\varpi_s,\varpi_s) - (\alpha_x,\alpha_x)} \id, & &  \widecheck{P} := \widecheck{R} - q^{(\varpi_s,\varpi_s)}\id.
\end{align*} 
Second are the {\em anti-holomorphic relations}
\begin{align} \label{eqn:aholoCALCRELS}
\widehat{P}_{12} \acute{R}_{23} \adel \mathbf{z} \wedge \adel \mathbf{z} = 0, & & 
\widecheck{Q}_{34} \acute{R}_{23} \adel \mathbf{z} \wedge \adel \mathbf{z} = 0,
\end{align}
where  we have again used leg notation, and have denoted
\begin{align*}
\widehat{P} := \widehat{R} - q^{(\varpi_s,\varpi_s)}\id , & & \widecheck{Q} := \widecheck{R} + q^{(\varpi_s,\varpi_s) - (\alpha_x,\alpha_x)} \id. 
\end{align*} 
Finally, we have the {\em cross-relations}
\begin{align} \label{eqn:crossCALCRELS}
\adel \mathbf{z} \wedge \del \mathbf{z} = - q^{- (\alpha_x,\alpha_x)} \, T^-_{1234} \del \mathbf{z} \wedge \adel \mathbf{z} + q^{(\varpi_s,\varpi_s)-(\alpha_x,\alpha_x)}z C_{12}T^{-}_{1234}\del \mathbf{z} \wedge \adel \mathbf{z},
\end{align}
where we have again used leg notation, and have denoted
\begin{align*}
T^{-}_{1234} := \grave{R}^-_{23}\widehat{R}^{-}_{12}\widecheck{R}_{34}\acute{R}_{23}, & & C_{kl} := \sum_{i = 1}^{\dim(V_{\varpi_s})} \grave{R}^{-ii}_{~\,\,kl}.
\end{align*}

Because of the highly technical nature of the relations, we find it helpful to highlight exactly which of their properties are used below. First, we note that when $q=1$, the relations reduce to the standard anti-commutative Grassmann relations. The second relevant property is that the commutation relations when $q \neq 1$ are generated by certain linear combinations of $2$-forms of the type 
$
\del z_{ab} \wed \del z_{ab}, \, \adel z_{ab} \wed \del z_{ab}, \, \del z_{ab} \wed \adel z_{ab}$, and $\adel z_{ab} \wed \adel z_{ab}, \text{for } a,b \in J,
$
with coefficients  Laurent polynomials in $q$, assuming that the chosen basis of $V_{\varpi_s}$ is a Laurent basis.

\subsection{Noncommutative Complex Structures}

When the Podle\'s calculus for $\O_q(S^2)$ was originally introduced in \cite{PodlesCalc}, it  was demonstrated to be a $*$-calculus. It was subsequently observed in  \cite[Proposition 3.4]{MMF2} that $\Omega_q^{\bullet}(\mathbb{CP}^n)$ is a $*$-calculus. The general result, for all  irreducible quantum flag manifolds, was later established in \cite[Theorem 4.2]{MarcoConj}. As observed in \cite{MarcoConj},  each $*$-calculus $\Omega^{\bullet}_q(G/L_S)$ carries a natural $\O_q(G)$-covariant complex structure $\Omega^{(\bullet,\bullet)}$. Moreover, since $\Phi(\Omega^1)$ decomposes into a direct sum of two non-isomorphic irreducible $U_q(\frak{l}_S)$-modules, this complex structure and its opposite are the only such complex structures.

Consider the subset of  $J$ given by
\begin{align*}
J_{(1)} := \{ i \in J \,| \, (\varpi_s, \varpi_s - \alpha_x - \textrm{wt}(v_i)) = 0\},
\end{align*}
and denote $M := |J_{(1)}|$. It was shown in \cite[Proposition 3.6]{HKdR} that bases of $\Phi(\Omega\hol)$ and $\Phi(\Omega\ahol)$ are given \mbox{respectively by} 
\begin{align} \label{eqn:tangentspaceBasis}
\left\{e^+_i := [\del z_{i\textrm{\tiny$M$}}] \,|\, \textrm{ for } i \in J_{(1)}\right\},  & & \left\{e^-_i := [\adel z_{Mi}] \,|\, \textrm{ for } i \in J_{(1)}\right\}\!.
\end{align}
We call a subset $\{k_1, \ldots, k_a\} \sseq \{1,\ldots, M\}$ {\em ordered} if $k_1 < \cdots < k_a$, and we denote by $O(a)$ all ordered subsets containing $a$ elements. For any two ordered subsets $K,L \sseq  \{1,\ldots, M\}$, we denote 
\begin{align*}
e^+_K \wed e^-_L := e^+_{k_1} \wed \cdots \wed e^+_{k_a} \wed  e^-_{l_1} \wed \cdots \wed e^-_{l_{a'}}.
\end{align*}
As shown in \cite[\textsection 3.3]{HKdR}, a basis of $\Phi(\Om^{(a,b)})$ is given by
\begin{align} \label{eqn:BASIS}
\Theta := \left\{ e^+_K \wed e^-_L \, | \, K \in O(a), \, L \in O(b)\right\}\!.
\end{align}

\begin{eg}
Let us now focus on quantum projective space $\O_q(\mathbb{CP}^n)$. The basis of $\Phi(\Omega^{1}_q(\mathbb{CP}^n))$ reduces to 
\begin{align*}
e^+_i = [\del z_{in}], & & e^-_i = [\adel z_{ni}], & & \textrm{ for } i = 2, \dots, n+1.
\end{align*}
The relations of the quantum exterior algebras $\Phi(\Omega^{(\bullet,0)})$ and $\Phi(\Omega^{(0,\bullet)})$ reduce to the standard quantum affine space, and its dual, respectively:
\begin{align*}
e^+_j \wedge e^+_i = - q e^+_i \wedge e^+_j, & & e^-_j \wedge e^-_i = - q^{-1} e^-_i \wedge e^-_j, & & \text{ for }1 \leq  i \leq j \leq n.
\end{align*}
\end{eg}

\begin{eg}
Outside of the quantum projective space case, the quantum exterior algebras exhibit a greater degree of noncommutativity. For example, let us consider the $B$-series  odd quantum quadrics $\O_q(\mathbf{Q}_{2n+1})$, for $n > 1$. As observed in \cite[\textsection 6]{HK}, the anti-holomorphic algebra  $\Phi(\Omega^{(0,\bullet)})$ is isomorphic to quantum orthogonal vector space $O^{2n-1}_q(\mathbb{C})$ originally considered by Faddeev, Reshetikhin, and Takhtajan in \cite{FRT89}.  
This algebra contains degree $1$ elements $y$ and $y'$ which do not skew-commute, that is 
\begin{align*}
y \wedge y' \neq c \, y' \wedge y, & & \textrm{ for any } c \in \mathbb{C}.
\end{align*}
More surprisingly, there exists a degree $1$ element $y_0 \in O^{2n-1}_q(\mathbb{C})$ such that
$$
y_0 \wedge y_0 \neq 0.
$$ 
\end{eg}

\subsection{Noncommutative K\"ahler Structures}

As shown in \cite[Theorem 4.6]{DOKSS}, for each irreducible quantum flag manifold $\O_q(G/L_S)$, there exists a real left $\O_q(G)$-coinvariant $(1,1)$-form $\kappa$, uniquely defined up to  real multiple. By extending the representation theoretic argument given in  \cite[\textsection 4.4]{MMF3} for the case of $\O_q(\mathbb{CP}^n)$, the form $\kappa$ is readily seen to be a closed  central element of  $\Om^{\bullet}$. In more detail, a direct examination confirms that the  $\O_q(L_S)$-comodules 
\begin{align*}
\Phi(\Omega^{(2,1)}) \simeq \Phi(\Omega^{(2,0)}) \otimes \Phi(\Omega^{(0,1)}), & & \text{ and } & & \Phi(\Omega^{(1,2)}) \simeq \Phi(\Omega^{(1,0)}) \otimes \Phi(\Omega^{(0,2)}),
\end{align*}
do not contain a copy of the trivial comodule. Hence, there can be no non-trivial map from $\mathbb{C}\kappa = \,^{\co(\O_q(G))}(\Omega^{(1,1)})$ to either $\Omega^{(2,1)}$ or $\Omega^{(1,2)}$, implying that $\exd \kappa = 0$.

For the special case of $\O_q(\mathbb{CP}^n)$, the pair $(\Om^{(\bullet,\bullet)},\k)$ was shown to be a K\"ahler structure in \cite[\textsection 4.4]{MMF3}, for all $q \in \mathbb{R}_{>0}$. Moreover, $(\Om^{(\bullet,\bullet)},\k)$ was shown to be positive definite for all $q$ sufficiently close to $1$. The K\"ahler structure requirements were later extended in \cite{MarcoConj}, for the explicit choice of K\"ahler form 
\begin{align} \label{eqn:kahlerform}
\kappa := \mathbf{i} \sum_{i,j,k =1}^{ \dim(V_{\varpi_x})} q^{(2\rho, \mathrm{wt}(v_i))} z_{ij} \exd z_{jk} \wedge \exd z_{ki}, 
\end{align}
where $\rho$ is the half-sum of positive roots of $\frak{g}$. It was shown that the pair $(\Om^{(\bullet,\bullet)},\k)$ is a covariant K\"ahler structure for all $q \in \mathbb{R}_{>0}\bs F$, where $F$ is a possibly empty finite subset of $\mathbb{R}_{>0}$. We denote the associated metric by $g_{\kappa}$, and call it the \emph{quantum Fubini--Study metric}.


\subsection{Metric Positivity and Closure of the Integral}  \label{subsection:CQHKahlerGLS}
 
In this subsection we address the main technical challenge to the applying spectral gap framework, which is to say, we establish positivity of the metric associated to the K\"ahler form $\kappa$, and verify that the associated integral is closed. While we will not need to recall the general form of the Hodge map, we will need to recall the following explicit identity, which follows directly from \cite[Definition 4.11]{MMF3}: For  any two forms $\omega,\nu \in \Omega^{1}$, it holds that 
\begin{align} \label{eqn:Hodge}
[g_{\kappa}(\omega, \nu)] = \frac{\mathbf{i}}{(M-1)!}\e' \circ \Phi(\ast_{\kappa})\left([\omega] \wed [\kappa]^{M-1} \wed [\nu] \right)\!,
\end{align}
where $\wedge$ on the left hand side denotes the multiplication induced on $\Phi(\Omega^{\bullet})$ by the multiplication of $\Omega^{\bullet}$ through Takeuchi's monoidal equivalence.

\begin{lem} \label{lem:commONB}
For $q=1$ the pairing $(\cdot,\cdot)$ is an inner product with respect to which $\Theta$ (as defined in \eqref{eqn:BASIS}) is an orthonormal set. 
\end{lem}
\begin{proof}
Note that when $q=1$, the $R$-matrix associated to any $U_q(\frak{g})$-module reduces to the identity matrix. This implies that the relations in \textsection \ref{subsection:gensandrels} reduce to the usual exterior algebra relations, giving us the isomorphism 
\begin{align} \label{eqn:qequal1Lambda}
\Lambda^{\bullet}\!\left(\Phi(\Omega^1)\right) \simeq \Phi(\Omega^{\bullet}(G/L_S)),
\end{align}
where $\Lambda^{\bullet}\!\left(\Phi(\Omega^1)\right)$ denotes the usual exterior algebra of $\Phi(\Omega^1)$. From the explicit presentation of the K\"ahler form given in \eqref{eqn:kahlerform}, taken together with \eqref{eqn:tangentspaceBasis}, we see that 
$$
[\kappa] = \sum_{i \in J_{(1)}} e^+_i \wedge e^-_i.
$$
A direct calculation now confirms that the basis elements $e^{\pm}_i,e^{\pm}_j \in \Theta$ are orthonormal with respect to the scalar product. Recalling now that the decomposition 
\begin{align*}
\Phi(\Omega^1) = \Phi(\Omega^{(1,0)}) \oplus \Phi(\Omega^{(0,1)})
\end{align*}
is orthogonal with respect to the scalar product $(\cdot,\cdot)$, we see that the classical Weil formula \cite[Theorem 1.2.31]{HUY} implies that $(\cdot,\cdot)$ coincides with the standard extension of $(\cdot,\cdot)$ to a sesquilinear form on the exterior algebra $\Lambda^{\bullet}\!\left(\Phi(\Omega^1(G/L_S))\right)$. In particular, $(\cdot,\cdot)$ is an inner product with respect to which $\Theta$ forms an orthonormal basis. 
\end{proof}

\begin{prop} \label{prop:posdefkappa}
For every irreducible quantum flag manifold $\O_q(G/L_S)$,  there exists an open interval $I$ around  $1$ such that $g_{\k}$ is positive definite, for all $q \in I$.
\end{prop}
\begin{proof} 
By Proposition \ref{prop:covHermStructures},  $g_{\k}$ is  positive definite if and only if an inner product is given by
\begin{align*}
(\cdot,\cdot): \Phi(\Omega^{\bullet}) \times \Phi(\Omega^{\bullet}) \to \mathbb{C}, & & ([\omega],[\nu]) \mapsto \e(g_{\k}(\omega,\nu)).
\end{align*}
We will prove the proposition by finding an interval $I$ around $1$ such that $(\cdot,\cdot)$ is positive definite whenever $q \in I$.

Since the commutation relations of $\Omega^{\bullet}$ (as presented in \textsection  \ref{subsection:gensandrels}) have $R$-matrix entry coefficients, the commutation relations of $\Phi(\Omega^{\bullet})$ must be generated by linear combinations of the degree $2$ basis elements with $R$-matrix entry coefficients. In particular, taking the basis of the fundamental representation $V_{\varpi_s}$ to be a Laurent basis, these  coefficients will be Laurent polynomials in $q$. Let $F$ be the (possibly empty) finite set of real numbers for which $(\Omega^{(\bullet,\bullet)},\kappa)$ is not a K\"ahler structure, and write $I_0$ for the largest open interval around $1$ which does not contain an element of $F$. For any two basis elements $e_{\alpha},e_{\beta} \in \Theta$, consider  the functional 
\begin{align*}
f_{\alpha \beta}:I_0 \to \mathbb{C}, & & q \mapsto (e_{\alpha}, e_{\beta}).
\end{align*}
Since the coefficients of the relations are Laurent ploynomials in $q$, the function $f_{\alpha\beta}$ is a Laurent polynomial in $q$, for all $e_{\alpha},e_{\beta}  \in \Theta$. 

Consider now the real vector space $\Phi(\Omega^{\bullet})_{\mathbb{R}}$ spanned by the basis elements of $\Theta$. For some $q \in I_0$, and a general element $\sum_{e_{\alpha} \in \Theta} c_{\alpha} e_{\alpha} \in \Phi(\Omega^{\bullet})_{\mathbb{R}}$, we have that 
\begin{align*}
\left(\sum_{e_{\alpha} \in  \Theta} c_{\alpha} e_{\alpha}, \sum_{e_{\beta} \in \Theta} c_{\beta} e_{\beta}  \right) = & \sum_{e_{\alpha},e_{\beta} \in \Theta} c_{\alpha} c_{\beta} \left( e_{\alpha}, e_{\beta} \right) \\
\geq &   \sum_{e_{\alpha} \in  \Theta} c_{\alpha}^2 \left( e_{\alpha}, e_{\alpha} \right) -  \sum_{e_{\alpha} \neq e_{\beta} \in \Theta} |c_{\alpha} c_{\beta} \left( e_{\alpha}, e_{\beta} \right)| \\
= &  \sum_{e_{\alpha} \in  \Theta} c_{\alpha}^2 f_{\alpha \alpha}(q) -  \sum_{e_{\alpha} \neq e_{\beta} \in \Theta} |c_{\alpha} c_{\beta} f_{\alpha\beta}(q)|.
\end{align*}
By Lemma \ref{lem:commONB} we know that $f_{\alpha\alpha}(1) > 0$ and $f_{\alpha \beta}(1) = 0$, since both functions are Laurent polynomials, they are necessarily continuous. This implies that for a sufficiently small interval $I \sseq I_0$ around $1$, we have
\begin{align*}
 \sum_{e_{\alpha} \in \Theta} c_{\alpha}^2 f_{\alpha \alpha}(q) -  \sum_{e_{\alpha} \neq e_{\beta} \in \Theta} |c_{\alpha} c_{\beta} f_{\alpha \beta}(q)| > 0, & & \textrm{ for all } q \in I.
\end{align*}
Thus we see that $(\cdot,\cdot)$ is positive definite on $\Phi(\Omega^{\bullet})_{\mathbb{R}}$. Consequently, $(\cdot,\cdot)$ extends to a positive definite map on $\Phi(\Omega^{\bullet})$, meaning that $g_{\k}$ is positive definite, for all $q \in I$.
\end{proof}

With positivity in hand, we are now ready to address the question of closure of the integral. To do so we follow the approach of \cite{MMF3} and reduce the problem to a question in representation theory.

\begin{prop} \label{thm:CQHKQuantumflags}
For each irreducible quantum flag manifold $\O_q(G/L_S)$, such that $q \notin F$, the integral $\int$ associated to the K\"ahler structure $(\Omega^{(\bullet,\bullet)},\kappa)$ is closed.
\end{prop}
\begin{proof}
We will prove the proposition by showing that $\Phi(\Omega^{(0,1)})$ does not contain a copy of the trivial $\O_q(L_S)$-comodule, and then appealing to \cite[Corollary 4.14]{MMF3}. Note that since the case of $\O_q(\mathbb{CP}^n)$ has been been dealt with in \cite[Lemma 3.4.4]{MMF3}, it follows from the dimensions presented in Table \ref{table:2} below that we can restrict our attention to those irreducible quantum flag manifolds for which  $\Phi(\Omega^{(0,1)})$ has dimension strictly greater than $1$. As also presented in Table \ref{table:2}, $\Phi(\Omega^{(0,1)})$ is irreducible as a $U_q(\frak{l}_S)$-module. Hence $\Phi(\Omega^{(0,1)})$ cannot contain a copy of the trivial comodule, implying that the integral is closed. 
\end{proof}

\subsection{Line Modules over the Irreducible Quantum Flag Manifolds}

In this subsection we prove one of the principal results of the paper, showing that the powerful tools of classical complex geometry allow us to prove general results about the spectral behaviour of their $q$-deformed differential operators. 

By Takeuchi's equivalence, the relative line modules over $\O_q(G/L_S)$ are indexed by the one-dimensional representations of $U_q(\frak{l}_S)$. These are in turn indexed by the integers $\mathbb{Z}$, or more explicitly, can be identified with the  weights $\mathbb{Z}\varpi_s \subseteq \mathcal{P}$. We denote the one-dimensional representations by $V_k$, and the corresponding line module $\EE_k$. (See \cite{PPACROB} for a more detailed presentation.)  

Each $\mathcal{E}_k$ possesses a unique covariant $(0,1)$-connection, which we denote by $\adel_{\mathcal{E}_k}$. Moreover, each $\adel_{\mathcal{E}_k}$ is flat and hence forms a covariant holomorphic structure for $\mathcal{E}_k$. Each $V_k$ clearly admits a covariant inner product, unique up to real scalar multiple. Thus it follows from Proposition \ref{prop:covHermStructures} that each $\EE_k$ admits a covariant Hermitian structure, unique up to real scalar multiple. We denote the associated quantum Chern connection by
$$
\nabla:\EE_k \to \Omega^1_q(G/L_S) \otimes \EE_k.
$$
As observed in \cite{DOKSS}, every line bundle over $\O_q(G/L_S)$ must be positive, flat, or negative. Combining this observation with the quantum  Borel--Weil theorem for the irreducible quantum flag manifolds, it was shown in \cite[Theorem 4.9]{DOKSS} that, for all $k \in \mathbb{Z}_{>0}$, we have $\mathcal{E}_k > 0$ and $\mathcal{E}_{-k} < 0$.  We denote by $\theta_k$, the strictly positive real number defined by \begin{align*}
\nabla^2(e) = -\mathbf{i}\theta_k e, & & \textrm{ for all  } e \in \EE_{-k}.
\end{align*}
It now follows from Theorem \ref{thm:lowerbound} that we have spectral gaps for our twisted Laplace and Dolbeault--Dirac operators.

\begin{thm}\label{thm:thethm}
Let $\O_q(G/L_S)$ be an irreducible quantum flag manifold, with $q \in I$, and let $\EE_{-k}$ be a relative line module, for $k \in \mathbb{Z}_{>0}$. For the Laplace operator
\begin{align} 
\Delta_{\adel_{\F}}:  \Omega^{(0,\bullet)} \to  \Omega^{(0,\bullet)},
\end{align}
a lower bound for the eigenvalues is given by $\theta$. For the Dolbeault--Dirac operator
\begin{align} 
D_{\adel_{\F}}: \Omega^{(0,\bullet)} \to  \Omega^{(0,\bullet)},
\end{align}
a lower bound for the absolute value of its eigenvalues is given by  $\sqrt{\theta}$.
\end{thm}

\begin{eg}
For the case of the Podle\'s sphere $\O_q(S^2)$ the curvature has been explicitly calculated, see \cite[Example 5.23]{BeggsMajid:Leabh} for details. In the conventions of this paper, for any line module $\EE_k$ over $\mathcal{O}_q(S^2)$, it holds that 
\begin{align*}
\nabla^2(e) =  -(k)_{q^{-2}} \mathbf{i} \kappa \otimes e, & & \textrm{ for all } e \in \mathcal{E}_k,
\end{align*}
where the quantum integer is given explicitly by
$
(k)_{q^{-2}} := 1 + q^{-2} + q^{-4} + \cdots + q^{-2(k-1)},
$
where we have scaled the K\"ahler form $\kappa$ so that
\begin{align} \label{eqn:firstCPNChern}
\nabla^2(e) = - \mathbf{i} \kappa \otimes e, & & \text{for all }   e \in \EE_1.
\end{align}
Thus we see that the absolute value of the eigenvalues of the $\EE_{-k}$-twisted Dolbeault--Dirac operators over $\O_q(S^2)$ are strictly bounded below by $\sqrt{(k)_{q^{-2}}}$. In particular, we see that twisting the Dolbeault--Dirac operator by $\EE_{-k}$, for an arbitrarily large $k$, produces an arbitrarily large spectral gap.

In general, the deformation of geometric integer quantities  to $q$-intergers is a ubiquitous  phenomenon in the differential calculus approach to the noncommutative geometry of quantum groups. This is in contrast to the isospectral approach  \cite{ConnesLandi}, where one sets out a classical spectrum in advance, and then builds a noncommutative geometry around it (see \cite{ VarillyvSDL,VarillyvSDLLocalIndex,NeshTus} for examples).  
\end{eg}

\subsection{Comparing Dolbeault Cohomology and Cyclic Cohomology}

Cyclic cohomology $HC^k$, as independently introduced by Connes \cite{Connes} and Tsygan \cite{TsyganCC}, is the standard replacement for de Rham cohomology in noncommutative geometry. However, when applied to quantum group examples it fails to preserve classical dimension. This phenomonen  is informally known as {\em dimension drop}, and is regarded by many as an unpleasant feature of the theory. For example, it was shown by Masuda, Nakagami, and  Watanabe \cite{MasudaSU2} that the cyclic cohomology of $\O_q(SU_2)$ satisfies $HC^3(\O_q(SU_2)) = 0$. This work was extended by Feng and Tsygan \cite{TsyganFeng}, who computed the cyclic cohomology of each Drinfeld--Jimbo coordinate algebra $\O_q(G)$. They showed that $HC^k(\O_q(G)) = 0$, for all $k$ greater than the rank of $G$. Vanishing of cohomology occurs even at the level of the quantum flag manifolds. For the simplest case, which is to say the Podle\'s sphere, its cyclic cohomology satisfies $HC^2(\O_q(S^2)) = 0$ \cite{MasudaPodles}.

We now observe that it follows from  Proposition \ref{prop:posdefkappa} and the hard Lefschetz theorem (as presented in \textsection \ref{subsection:hardLef}) that the dimension drop phenomenon does not occur for the de Rham cohomology of the Heckenberger--Kolb calculi. This proposes it as a more natural cohomology theory for the irreducible quantum flag manifolds.

\begin{thm}
For any irreducible quantum flag manifold $\O_q(G/L_S)$, such that  $q \in I$, the  de Rham cohomology $H^{\bullet}$ of $\Omega^{\bullet}_q(G/L_S)$ satisfies 
\begin{align*}
H^{2k} \neq 0, & & \textrm{ for all } k = 1, \dots, M = \frac{1}{2}\dim\left(\Omega^{\bullet}_q(G/L_S)\right)\!.
\end{align*}
\end{thm}

\begin{remark}
Twisted cyclic cohomology was introduced in \cite{KMT} as an attempt to address the unpleasant dimension drop of cyclic cohomology. It generalises cyclic cohomology through the introduction of an algebra automorphism $\sigma$, which when $\sigma = \id$ reduces to ordinary cyclic cohomology. For $\O_q(SU_n)$, with $\sigma$ choosen to be the modular automorphism of the Haar state \cite[\textsection 11.3.4]{KSLeabh}, the dimension of the twisted cyclic cohomology coincides with the classical dimension \cite{UKHadfield}. Analagous results were obtained for the  Podle\'s sphere in \cite{HadfieldPodles}.  The relationship between twisted cyclic cohomology and the cohomology of the Heckenberger--Kolb calculi is at present unclear.
\end{remark}

\subsection{Orthogonality of the Degree $1$ Basis Elements}

Direct investigation of low-dimensional cases \cite[\textsection 5.4]{MMF3} suggests that $g_{\k}$ will be positive definite for all $q \in \mathbb{R}_{>0}$. However, verifying this appears to be a difficult problem, most likely requiring the introduction of new structures and ideas. Here we content ourselves with showing that the degree $1$ elements, of the basis $\Theta$ defined in \eqref{eqn:BASIS}, are orthogonal. We do this using a weight argument which requires us to first recall some facts about the $U_q(\frak{l}_S)$-module structures of the cotangent spaces of the quantum flag manifolds originally observed in \cite[\textsection 6]{HK}.

Classically the algebra $\frak{l}_S$ is reductive, and hence decomposes into a direct sum $\frak{l}_S^s \oplus \frak{u}_1$, comprised of a semisimple part and a commutative part, respectively. In the quantum setting, we are thus motivated to consider the subalgebra 
 \begin{align*}
U_q(\frak{l}_S^{\,\mathrm{s}}) := \big< K_i, E_i, F_i \,|\,  i \in S \big> \sseq U_q(\frak{l}_S).
\end{align*}  
The  table immediately above presents $\Phi(\Omega^{(1,0)})$ and  $\Phi(\Omega^{(0,1)})$ as $U_q(\frak{l}^{\mathrm{\,s}}_S)$-modules and gives their dimensions.

An important point to note is that the weight spaces of the cotangent spaces $\Phi(\Omega^{(1,0)})$ and $\Phi(\Omega^{(0,1)})$ are all one-dimensional,  as can be deduced from the Weyl character formula for a complex semisimple Lie algebra $\frak{g}$. Alternatively, an explicit presentation of those irreducible $\frak{g}$-modules $V_{\lambda}$, for $\lambda \in \mathcal{P}^+$, whose weight spaces are one-dimensional can be found in \cite[Chapter 6]{Seitz}.

\begin{center} 
\captionof{table}{Irreducible Quantum Flag Manifold Cotangent Spaces: presenting the semisimple subalgebra $U_q(\frak{l}_S^{\mathrm{\,s}})$, the holomorphic and anti-holomorphic cotangent spaces described as $U_q(\frak{l}_S^{\mathrm{\,s}})$-modules, and the dimension $M$ of the both spaces, or equivalently the dimension of $G/L_S$ as a complex manifold.} \label{table:2}
\begin{tabular}{|c|c|c|c|c|c| }

\hline

& & &  & \\

$\O_q(G/L_S)$ & $U_q(\frak{l}^{\mathrm{\,s}}_S)$ & $\Phi(\Omega^{(0,1)})$ & $\Phi(\Omega^{(1,0)})$ & dimension $M$ \\

& & &  & \\

\hline

& & &  & \\

 $\O_q(\text{Gr}_{k,n+1})$  & $U_q(\frak{sl}_k \oplus \frak{sl}_{n-k+1})$ & $V_{\varpi_1} \otimes V_{\varpi_{1}}$ & $V_{\varpi_{k-1}} \otimes V_{\varpi_{n-k}}$ & $k(n\!-\!k\!+\!1)$ \\

 &  & &    & \\

  $\O_q(\mathbf{Q}_{2n+1})$ &  $U_q(\frak{so}_{2n-1})$ & $V_{\varpi_{1}}$ & $V_{\varpi_{1}}$ & $2n-1$  \\

 & &  & &   \\

  $\O_q(\mathbf{L}_{n})$   & $U_q(\frak{sl}_{n})$ & $V_{2\varpi_{1}}$ & $V_{2\varpi_{n-1}}$ &  $\frac{n(n+1)}{2}$  \\

  &  &   &    & \\ 

 $\O_q(\mathbf{Q}_{2n})$  & $U_q(\frak{so}_{2(n-1)})$ & $V_{\varpi_1}$ & $V_{\varpi_{1}}$ &  $2(n-1)$ \\ 
 &  &   &  &  \\

  $\O_q(\textbf{S}_{n})$ & $U_q(\frak{sl}_{n})$ & $V_{\varpi_2}$ & $V_{\varpi_{n-2}}$ & $\frac{n(n-1)}{2}$ \\
 &  &  &  & \\

  $\O_q(\mathbb{OP}^2)$  & $U_q(\frak{so}_{10})$ & $V_{\varpi_6}$ & $V_{\varpi_5}$  &  16\\
 &   & &   & \\ 

 $\O_q(\textbf{F})$    & $U_q(\frak{e}_6)$ 
& $V_{\varpi_1}$ & $V_{\varpi_6}$  & 27 \\

&  &  &   & \\ 

\hline
\end{tabular}
\end{center}

\begin{prop}
The degree $1$ elements of $\Theta$ are orthogonal, for all $q \notin F$.
\end{prop}
\begin{proof}
It follows from \eqref{eqn:Hodge}  that, for any two basis elements $e^+_{i}, \, e^+_{j} \in \Phi(\Omega^{(1,0)})$, we have 
\begin{align*}
(e^+_{i},e^+_{j}) =  \frac{\mathbf{i}}{(M-1)!}\e' \circ \Phi(\ast_{\kappa})\left(e^+_{i} \wed [\kappa]^{M-1} \wed e^+_{j} \right)\!.
\end{align*}
As noted above, it follows from the presentation  in Table \ref{table:2} of $\Phi(\Omega^{(1,0)})$ and $\Phi(\Omega^{(0,1)})$ as $U_q(\frak{l}^{\mathrm{\,s}}_S)$-modules that their weight spaces are all one-dimensional. This implies that the elements $e^+_{i} = [\del z_{iM}]$ have distinct weights, and that 
$$
\mathrm{wt}(e^+_{i}) = \mathrm{wt}[\overline{\partial}z_{iM}] = - \mathrm{wt}[\adel z_{Mi} ] = - \mathrm{wt}(e^-_{i}).
$$
Thus any product $e^+_{i} \wed e^-_{j}$ will have weight zero if and only if $i=j$. Since the K\"ahler form $\kappa$ is left $\O_q(G)$-coinvariant,  the product $e^+_{i} \wed [\kappa]^{M-1} \wed e^+_{j}$ will have degree $0$ if and only if $i=j$. However, this element lives in $\Phi(\Omega^{(M,M)})$ which is trivial as a right $U_q(\frak{l}_S)$-module. Thus if $i \neq j$, we must have that  $(e^+_{i},e^+_{j}) = 0$. An analogous argument shows that the set of basis elements contained in $\Phi(\Omega^{(0,1)})$ is orthogonal. Finally, since $\Phi(\Omega^{(1,0)})$ and $\Phi(\Omega^{(0,1)})$ are orthogonal spaces, we see that the set of all degree $1$ basis elements is orthogonal.
\end{proof}

\appendix

\section{Quantum Homogeneous Spaces and Takeuchi's Equivalence} \label{app:TAK}

In this appendix we recall  Takeuchi's equivalence \cite{Tak} for quantum homogeneous spaces in the form most suited to our purposes. Just as for the rest of the paper, $A$ will always denote a Hopf algebra defined over $\mathbb{C}$, with coproduct, counit, antipode denoted by $\Delta,\epsilon$, and $S$ respectively, without explicit reference to $A$.

\subsection{Relative Hopf Modules} 

For any subalgebra $B \subseteq A$, we say that $A$ is {\em faithfully flat} as a right $B$-module if the functor  $A \oby_B -:\,_B\mathrm{Mod} \to \mathrm{Vect}_{\mathbb{C}}$, from the category of left $B$-modules to the category of complex vector spaces, preserves and reflects exact sequences. We say that a left coideal subalgebra $B \sseq A$ is a \emph{quantum homogeneous $A$-space} if $A$ is faithfully flat as a right $B$-module and we have $B^+A = AB^+$, where $B^+ := \ker(\e_B)$. 
It follows from \cite[Theorem 1]{Tak} that, for the Hopf algebra surjection $\pi_B:A \to A/B^+A$, and the associated right $\pi_B(A)$-coaction $\Delta_{R,\pi_B} := (\id \otimes \pi_B) \circ \Delta$, the space of coinvariants is equal to $B$, that is
$$
B = A^{\co(A/B^+A)} := \{a \in A \,|\, \Delta_{R,\pi_B}(b) = b \otimes 1\}.
$$

We denote by~${}^A_B\mathrm{mod}$ the category of \emph{(finitely generated) relative Hopf modules}, that is, the category whose objects are left \mbox{$A$-comodules} \mbox{$\DEL_L:\mathcal{F} \to A \otimes \mathcal{F}$}, endowed with a finitely generated left $B$-module structure such that, for all $f \in \mathcal{F},$ \mbox{$b \in B$}, we have $\DEL_L(bf) = \Delta_L(b)\DEL_L(f)$, 
and whose morphisms are left $A$-comodule, left $B$-module, maps. Every relative Hopf module is automatically projective as a left $B$-module. 

\subsection{Takeuchi's Equivalence}
In the following we denote by ${}^{\pi_B}\mathrm{mod}$ the category whose objects are finite-dimensional left \mbox{$\pi_B(A)$-comodules}, and whose morphisms are left $\pi_B(A)$-comodule maps. We use similar notation for right $\pi_B(A)$-comodules. 

Define a functor $\Phi:{}^A_B\mathrm{mod} \to {}^{\pi_B}\mathrm{mod}$ by setting $\Phi(\mathcal{F}) := \mathcal{F}/B^+\mathcal{F}$, where we have denoted  $B^+ := B \cap \ker(\e)$, and the left $\pi_B(A)$-comodule structure of $\Phi(\mathcal{\F})$ is given by 
$
\Delta_L[f] := \pi_B(f_{(-1)})\otimes [f_{(0)}],
$
with square brackets denoting the coset of an element in $\Phi(\mathcal{\F})$. In the other direction, we use the cotensor product $\square_{\pi_B(A)}$, which we find convenient to denote by $\square_{\pi_B}$. Define a functor $\Psi: {}^{\pi_B}\mathrm{mod} \to {}^A_B\mathrm{mod}$ by setting $\Psi(V) := A \,\square_{\pi_B} V$, where the left $B$-module and left $A$-comodule structures of $\Psi(V)$ are defined on the first tensor factor, and if $\gamma$ is a morphism in ${}^{\pi_B}\mathrm{mod}$, then $\Psi(\gamma) := \id \otimes \gamma$. Note that $A$ is naturally an object in ${}^A_B\mathrm{mod}$, and that $\Phi(A) = \pi_B(A).$

As established in~\cite[Theorem 1]{Tak}, an adjoint equivalence of categories between~${}^A_B\mathrm{mod}$ and~${}^{\pi_B}\mathrm{mod}$, which we call \emph{Takeuchi's equivalence}, is given by the functors $\Phi$ and $\Psi$, the unit natural isomorphism
$
\unit: \F \to \Psi \circ \Phi(\F)$, defined by $\unit(f) = f_{(-1)} \otimes [f_{(0)}],
$
and the counit natural isomorphism 
$
\counit := \e \otimes \id: \Phi \circ \Psi(V) \to V.
$
The \emph{dimension} $\mathrm{dim}(\F)$ of an object $\F \in {}^A_B\mathrm{mod}$ is the vector space dimension of $\Phi(\F)$.

Consider now the category $^A_B\textrm{mod}_0$ whose objects are objects $\F$ in ${}^A_B\textrm{mod}$ endowed with the right $B$-module structure 
\begin{align*}
fb := f_{(-2)}bS(f_{(-1)})f_{(0)}, & & \textrm{ for } f \in \F, \, b \in B. 
\end{align*}
This give $\F$ the structure of a $B$-bimodule. It is instructive to note that the right action of  any $b \in B$ on $A \square_{\pi_B} \Phi(\F)$ is right multiplication on the first tensor factor. Clearly, $^A_B\textrm{mod}_0$ is equivalent to $^A_B\textrm{mod}$, however $^A_B\textrm{mod}_0$ comes equipped with a monoidal structure given by the tensor product $\otimes_B$. Moreover, with respect to the obvious monoidal structure on $^H\mathrm{mod}$, Takeuchi's equivalence is easily endowed with the structure of a monoidal equivalence (see \cite[\textsection 4]{MMF2}).  A \emph{relative line module} over $B$ is an invertible object $\EE$ in the category ${}^A_B\mathrm{mod}_0$. Indeed, an object $\EE$ in ${}^A_B\mathrm{mod}_0$ is a relative line module if and only if $\dim(\EE) = 1$.

\subsection{Conjugates and Duals}

We now discuss dual objects in  the categories ${}^A_B\mathrm{mod}_0$ and ${}^{\pi_B}\mathrm{mod}$. Since ${}^{\pi_B}\mathrm{mod}$ is a rigid monoidal category, $^A_B\mathrm{mod}_0$ is a rigid monoidal category. In particular,  every object $\F \in {}^A_B\mathrm{mod}_0$ admits a right dual ${}^{\vee}\!\F$. 

For a general $B$-bimodule $\F$, we can endow the right $B$-module ${}_{B}\mathrm{Hom}(\F,B)$ with a $B$-bimodule structure by setting $b\f(f) := \f(fb)$, for $\f \in {}^{\vee}\F, \, f \in \F, \, b \in B$. If $\F$ is projective as a left $B$-module (as for a relative Hopf module) then ${}_{B}\mathrm{Hom}(\F,B)$ is a right dual for $\F$. Now $^A_B\mathrm{mod}_0$ is a (non-full) monoidal subcategory of the category of $B$-bimodules $_B\mathrm{Mod}_B$ endowed with its usual tensor product $\otimes_B$.  Thus since right duals are unique up to isomorphism,  ${}^{\vee}\!\F$ must be isomorphic to ${}_{B}\mathrm{Hom}(\F,B)$ as a $B$-bimodule,  justifying the abuse of notation.

Let us now assume that $A$ is a Hopf $*$-algebra, $B$ is a $*$-subalgebra, and hence that $\pi_B(A)$ is a Hopf $*$-algebra. For any object $\F \in {}^A_B\textrm{mod}_0$, its \emph{conjugate}  is the $B$-bimodule  defined  to be the conjugate bimodule $\overline{\F}$ (as defined in \textsection \ref{subsection:covariantHS}) endowed with the left $A$-comodule structure 
${\overline{\F} \mto A \otimes \overline{\F}}$, defined by $\overline{f} = (f_{(-1)})^* \otimes \overline{f_{(0)}}$. As shown in \cite[Corollary 2.11]{OSV},  the conjugate of $\F$ is again an object in ${}^A_B\textrm{mod}_0$. It is instructive to note that  the corresponding operation on any object in $V \in {}^{\pi_B}\textrm{mod}$ is the complex conjugate of  $V$ coming from the  Hopf $*$-algebra structure of $\pi_B(A)$.

\section{A Remark on Quantum Spin-Dirac Operators} \label{subsection:QFMSpin}

In the classical setting every complex manifold has a canonically associated spin structure. Moreover, by a theorem of Atiyah, a $2m$-dimensional compact Hermitian manifold is spin if and only if its canonical line bundle $\Om^{(m,0)}$ admits a holomorphic square root $\sqrt{\Om^{(m,0)}}$ \mbox{\cite[Proposition 3.2]{AtiyahSpinSurfaces}}. Then the spin Dirac operator is isomorphic to the Dolbeault--Dirac operator twisted by $\sqrt{\Om^{(m,0)}}$.  The classical irreducible flag manifolds which are spin are given in the following table, see \cite[Appendix C]{DOKSS} for further details.

We note that in each case the square root is a negative line bundle. Thus it follows that Theorem \ref{thm:thethm} implies a spectral gap around $0$ for each $q$-deformed spin Dirac operator over an irreducible flag manifold. An interesting question to ask is if Friedrich's spectral estimates \cite[\textsection 5]{FriedrichDirac} can be extended to this setting.

\begin{center}
\captionof{table}{irreducible quantum flag manifolds, criteria for when the Hermitian manifold is spin, and the holomorphic square root of the canonical  line module} \label{table:CQFMsEk}
\begin{tabular}{|c|l|c| }
  \hline
    & & \\
~~  $\O_q(G/L_S)$ ~~ & ~~ Spin Criteria  ~~& ~~ $\sqrt{\Om^{(m,0)}}$  ~~ \\ 
  & & \\
  \hline
    & & \\
  $\O_q(\text{Gr}_{s,n+1})$ & \textrm{ spin, for all} $n \in 2\mathbb{Z}_{> 0} + 1$ &  $\EE_{-(n+1)/2}$ 
 \\ 
   & & \\
  $\O_q(\mathbf{L}_{n})$ & \textrm{ spin, for all} $n \in 2\mathbb{Z}_{> 0} + 1$ &  $\EE_{-(n+1)/2}$ 
 \\ 
   & & \\
  $\O_q(\mathbf{Q}_{2n})$ & \textrm{ spin, for all} $n \in \mathbb{Z}_{> 0}$ & $\EE_{-n+1}$ 
 \\ 
   & & \\
  $\O_q(\textbf{S}_{n})$ &  \textrm{ spin, for all } $n \in \mathbb{Z}_{>0}$ & $\EE_{-n+1}$\\
    & & \\
  $\O_q(\mathbb{OP}^2)$ & \textrm{ spin }& $\EE_{-6}$\\
    & & \\
  $\O_q(\textbf{F})$  &  \textrm{ spin } & $\EE_{-9}$\\
  & & \\
\hline
\end{tabular}
\end{center}

 \bibliographystyle{abbrv}

\end{document}